\documentclass[11pt,reqno]{amsart}

\usepackage{amsfonts}
\usepackage{amsopn}
\usepackage{amssymb}
\usepackage{graphicx}
\usepackage[margin=1in]{geometry}
\usepackage{epstopdf}
\usepackage{algorithmic}
\usepackage{ifpdf}
\usepackage{mathrsfs}
\usepackage{mathtools}
\usepackage{subcaption}
\usepackage{mathrsfs,euscript}
\usepackage{xcolor}
\usepackage{tikz}
\usetikzlibrary{arrows.meta}
\usepackage[hidelinks]{hyperref}
\usepackage[capitalize,nameinlink]{cleveref}

\ifpdf
  \DeclareGraphicsExtensions{.eps,.pdf,.png,.jpg}
\else
  \DeclareGraphicsExtensions{.eps}
\fi

\theoremstyle{plain}
\newtheorem{theorem}{Theorem}[section]
\newtheorem{lemma}[theorem]{Lemma}
\newtheorem{proposition}[theorem]{Proposition}

\theoremstyle{definition}

\crefname{hypothesis}{Hypothesis}{Hypotheses}

\numberwithin{equation}{section}

\title[Bulk-Surface PDE with Open Boundary]{Bulk-Surface Coupled PDE with an Open Boundary}
\author{Charles L. Epstein}
\address{Center for Computational Mathematics, Flatiron Institute of the Simons Foundation, 162 5th Ave., New York, NY 10010, USA}
\email{cepstein@flatironinstitute.org}
\author{Yoichiro Mori}
\address{Department of Mathematics, University of Pennsylvania, Philadelphia, PA 19104, USA}
\address{Department of Biology, University of Pennsylvania, Philadelphia, PA 19104, USA}
\email{y1mori@sas.upenn.edu}
\author{Han Zhou}
\address{Department of Mathematics, University of Pennsylvania, Philadelphia, PA 19104, USA}
\email{hzhou24@sas.upenn.edu}
\date{}

\hypersetup{
  pdftitle={Bulk-Surface Coupled PDE with an Open Boundary},
  pdfauthor={Charles L. Epstein, Yoichiro Mori, and Han Zhou}
}

\newcommand{\paren}[1]{\left(#1\right)}
\newcommand{\jump}[1]{\left[#1\right]}

\newcommand{\p}{\partial}

\newcommand{\at}[2]{\left. #1 \right|_{#2}}
\newcommand{\grad}[1]{\nabla #1}

\newcommand{\mc}[1]{\mathcal{#1}}
\newcommand{\wh}[1]{\widehat{#1}}
\newcommand{\wt}[1]{\widetilde{#1}}
\newcommand{\bm}[1]{\boldsymbol{#1}}
\newcommand{\abs}[1]{\left\lvert #1 \right\rvert}
\newcommand{\norm}[1]{\left\lVert #1 \right\rVert}
\newcommand{\inprod}[2]{\left( #1,#2 \right)}
\newcommand{\dual}[2]{\left\langle #1,#2 \right\rangle}

\newcommand{\starnorm}[1]{\left| #1 \right|_*}

\newcommand{\vph}{\varphi}
\newcommand{\rev}[1]{#1}
\newcommand{\cle}[1]{#1}

\newcommand{\R}{\mathbb{R}}
\newcommand{\C}{\mathbb{C}}
\newcommand{\N}{\mathbb{N}}
\newcommand{\Z}{\mathbb{Z}}
\newcommand{\T}{\mathbb{T}}

\begin{document}

\begin{abstract}
  We study a bulk–surface coupled Laplace system involving an embedded open boundary. The problem is reformulated as an integro-differential equation using boundary integral representations, for which we establish existence and uniqueness of the solution. A Wiener–Hopf technique is employed to study the solution regularity and derive asymptotic expressions for the edge singularity. Building on these results, we develop a finite element method that incorporates the singularity structure, and we provide a rigorous error analysis. Numerical experiments confirm the theoretical convergence rates.
\end{abstract}

\maketitle

\section{Introduction}
Let $\Omega_1$ be a simply connected bounded domain with boundary $\Gamma=\partial\Omega_1$ and $S$ be an open portion of $\Gamma$ with two endpoints $\partial S = \{p_1,p_2\}$.
Let $\bm n:\Gamma\to\R^2$ be the unit normal that points from $\Omega_1$ to $\Omega_2$.
Set $\Omega_2 = \R^2\setminus\overline{\Omega_1}$.
Let $\bm \nu$ be the outer normal vector on $\partial S$ \cle{ as a subset of $\Gamma.$}
\rev{The domains \(\Omega_1\) and \(\Omega_2\) are auxiliary and serve to orient the two sides of the open boundary \(S\), thereby fixing the sign convention for the normal \(\bm n\); see Figure~\ref{fig:setup} for a schematic illustration.}
Given $f:S\to \R$ and $g:\partial S\to \R$, we consider the coupled bulk-surface system
\begin{subequations}\label{eqn:arc-pde}
\begin{align}
    - \Delta u &= 0,\quad \text{ in }\mathbb{R}^2\backslash \overline{S},\label{eqn:arc-pde-a}\\
    u & = U, \quad \text{ on }S,\label{eqn:arc-pde-b}\\
    - \Delta_\Gamma U + \jump{\partial_{\bm n} u }&=  f ,\quad \text{ on }S,\label{eqn:arc-pde-c}\\
    \bm \nu\cdot \grad_\Gamma U &= g, \quad \text{ on }\partial S,\label{eqn:arc-pde-d}\\
    u(x) &= \mc O(|x|^{-1}),\quad |x| \to \infty,\label{eqn:arc-pde-e}
\end{align}    
\end{subequations}
where $u:\R^2\to\R$ and $U:S\to \R$ are the unknown solutions.
Here, $\grad_\Gamma$ is the surface gradient on $\Gamma$ and $\Delta_\Gamma = \grad_\Gamma\cdot\grad_\Gamma$ is the Laplace-Beltrami operator on $\Gamma$.
The notation $\jump{\partial_{\bm n} u}$ is the jump of the normal derivative across $S$, defined as
\begin{equation}
    \jump{\partial_{\bm n} u} (x)= \lim_{\varepsilon\to 0+}\bm n(x)\cdot (\grad u(x - \varepsilon \bm n) - \grad u(x + \varepsilon \bm n)),\quad x\in \Gamma.
\end{equation}
The functions $f$ and $g$ are assumed to satisfy the compatibility condition  
\begin{equation}\label{eqn:compatible}
    \int_S f(x)\, ds_x + \sum_{p \in \partial S} g(p) = 0,
\end{equation}
which ensures consistency with the decay of the solution at infinity.

\begin{figure}[t]
\centering
\begin{tikzpicture}[scale=1.05, every node/.style={font=\small,text=black}]
    \coordinate (p1) at ({(1.85+0.33*cos(125))*cos(25)}, {(1.15+0.23*cos(125))*sin(25)});
    \coordinate (p2) at ({(1.85+0.33*cos(725))*cos(145)}, {(1.15+0.23*cos(725))*sin(145)});
    \coordinate (npt) at ({(1.85+0.33*cos(375))*cos(75)}, {(1.15+0.23*cos(375))*sin(75)});
    \fill[blue!5] plot[domain=0:360, samples=220, smooth] ({(1.85+0.33*cos(5*\x))*cos(\x)}, {(1.15+0.23*cos(5*\x))*sin(\x)}) -- cycle;
    \draw[thick, dashed, gray!75] plot[domain=0:360, samples=220, smooth] ({(1.85+0.33*cos(5*\x))*cos(\x)}, {(1.15+0.23*cos(5*\x))*sin(\x)}) -- cycle;
    \draw[blue, line width=2.4pt] plot[domain=25:145, samples=100, smooth] ({(1.85+0.33*cos(5*\x))*cos(\x)}, {(1.15+0.23*cos(5*\x))*sin(\x)});
    \fill[blue] (p1) circle (1.5pt) node[above right, text=black] {$p_1$};
    \fill[blue] (p2) circle (1.5pt) node[above left, text=black] {$p_2$};
    \node at (0,-0.15) {$\Omega_1$};
    \node at (3.0,0.2) {$\Omega_2$};
    \node at (-0.65,1.55) {$S\subset\Gamma$};
    \node at (-1.65,-1.05) {$\Gamma=\partial\Omega_1$};
    \draw[-{Latex[length=2mm]}, thick] (npt) -- ++(0.12,0.58) node[above, text=black] {$\bm n$};
    \draw[-{Latex[length=2mm]}, blue, thick] (p1) -- ++(0.53,-0.14) node[right, text=black] {$\bm\nu$};
    \draw[-{Latex[length=2mm]}, blue, thick] (p2) -- ++(-0.38,-0.40) node[below left, text=black] {$\bm\nu$};
\end{tikzpicture}
\caption{\rev{Geometry and notation. The open arc \(S\) is a proper open subset of \(\Gamma=\partial\Omega_1\); the remaining part \(\Gamma\setminus\overline S\) is shown as a dashed curve. The regions \(\Omega_1\) and \(\Omega_2\) determine the orientation of \(\bm n\), and \(\bm\nu\) denotes the outward normal to \(S\) at its endpoints.}}
\label{fig:setup}
\end{figure}


Bulk--surface coupled partial differential equations (PDEs) arise in a wide range of physical, biological, and chemical processes where dynamics in a bulk domain interact with processes occurring on its boundary or on an embedded surface. Typical examples include chemical reactions and pattern formation on cell membranes coupled to cytoplasmic diffusion~\cite{Madzvamuse2015,Madzvamuse2016,Paquin-Lefebvre2019,Chechkin2012}, as well as coupled diffusion in crystalline systems~\cite{Flynn2006,Kosolobov2022}. In such systems, quantities diffuse within both the bulk and the surface and are coupled through interfacial fluxes or dynamic boundary conditions that govern exchange between the two regions. The system~\eqref{eqn:arc-pde} considered in this work represents the time-independent formulation of such a coupled diffusion problem.

There has been a growing body of theoretical and computational research on bulk--surface coupled PDEs involving closed surfaces in recent years~\cite{Hausberg2018,Wu2024,Zhang2025,Colli2019,Li2021,Altmann2021,Elliott2013,Madzvamuse2016,Tang2025}. To the best of our knowledge, the present paper is the first to address both the well-posedness and numerical approximation of coupled bulk--surface PDEs with an open surface.
Compared to problems with closed surfaces, systems involving open surfaces pose additional challenges in both analysis and numerical approximation. The domain is typically non-Lipschitz, and the solution lacks standard boundary regularity. For instance, in screen-type problems, the solution often exhibits singularities at the edge of the open boundary. Such singularities prevent the solution from attaining higher regularity even when the data are smooth.

There has been a long and systematic line of research on PDEs posed on open surfaces, particularly on screen problems in wave scattering and crack problems in continuum mechanics. The earliest rigorous studies on existence and uniqueness for the Helmholtz equation with an open boundary are due to Hayashi~\cite{Hayashi1973,Hayashi1977a,Hayashi1977b}. The Wiener--Hopf decomposition technique for analyzing boundary regularity in pseudodifferential equations on bounded domains and manifolds was developed by Eskin~\cite{eskin1981boundary}, providing a systematic framework for deriving asymptotic expressions of edge singularities. This technique was further applied to screen and mixed boundary value problems by Stephan~\cite{Stephan1986,Stephan1987a,Stephan1987b}, and to crack problems by Duduchava et~al.~\cite{Duduchava1995}. Subsequently, Costabel et~al.~\cite{Costabel2003} conducted a more refined analysis of the singularity structure and demonstrated that the logarithmic term predicted by the general theory is a ``shadow term'' that should not appear. More recent studies have extended these results to problems involving multiple or fractal screens~\cite{Jerez-Hanckes2018,Jerez-Hanckes2020,Chandler-Wilde2017}.

The remainder of the paper is organized as follows. Section~\ref{sec:prep} introduces the notation and functional spaces used throughout. In Section~\ref{sec:exist}, we establish the existence and uniqueness of weak solutions. Section~\ref{sec:reg} analyzes the regularity and provides the asymptotic expansion of the edge singularity. In Section~\ref{sec:fem}, we develop a finite element method that incorporates the identified singularities and provide a corresponding convergence analysis. \rev{Section~\ref{sec:conclusion} briefly summarizes the \cle{main results of the} 
paper.}

\section{Notation and functional spaces}\label{sec:prep}
Throughout this work, we focus on dimensions $n=1, 2$. Following \cite{eskin1981boundary}, we define the Fourier transform and its inverse as follows
\begin{align}
     (\mathcal{F}\vph)(\xi) = \int_{\R^n} \vph(x) e^{-i x\cdot\xi}\,dx ,\quad
     (\mathcal{F}^{-1}\wh \vph)(x) = \frac{1}{(2\pi)^n} \int_{\R^n} \wh\vph(\xi) e^{i x\cdot\xi}\,d\xi.
\end{align}
For $s\in\R$, the Sobolev space $H^s(\R^n)$ is defined via Fourier transform
\begin{equation}
    H^s(\R^n) = \{u \in \mathscr S'(\R^n): \int_{\R^2} (1 + |\xi|^2)^s |\wh u(\xi)|^2 \,d\xi < +\infty \}.
\end{equation}

Let $\Omega \subset \R^n$ be a non-empty open subset. We \cle{let}
$p_\Omega$ denote the operation of restriction to $\Omega,$ and
$u\mapsto \wt{u},$ \cle{ a bounded } extension operator,
$H^s(\Omega)\to H^s(\R^n),$ with $\norm{\wt u}_{H^s(\R^n)} \le C
\norm{u}_{H^s(\Omega)}$. \cle{For such an extension operator to exist
  $\partial\Omega$ must satisfy some regularity condition. In what
  follows we assume that $\Omega$ is either a half space, or is
  bounded and $\partial\Omega$ is locally Lipschitz, which suffices. Under this
  assumption $p_{\Omega}[C^{\infty}_c(\R^n)]$ is dense in
  $H^s(\Omega).$} When $\Omega = \R_+$, we write $p_{\R_+}$ as $p_+$,
and use $u_+$ to denote the extension of $u$ by zero outside $\R_+$.

For an integer $k
\ge 0$ and $\alpha \in (0,1]$, the $C^{k,\alpha}(\Omega)$ norm is
  defined as
\begin{equation}
    \|u\|_{C^{k,\alpha}(\Omega)}
    \coloneqq 
    \sum_{|\beta| \le k} \|D^\beta u\|_{C^{0}(\Omega)} 
    + \sum_{|\beta| = k} [D^\beta u]_{C^{0,\alpha}(\Omega)}.
\end{equation}
where the $C^0$-norm and the H\"older seminorm are given by
\begin{equation}
    \|f\|_{C^0(\Omega)} \coloneqq \sup_{x \in \Omega} |f(x)|, 
    \quad
    [f]_{C^{0,\alpha}(\Omega)} 
    \coloneqq 
    \sup_{\substack{x,y \in \Omega,  x \neq y}}
    \frac{|f(x) - f(y)|}{|x - y|^{\alpha}}.
\end{equation}
The H\"older space $C^{k,\alpha}(\Omega)$ is defined by  
\begin{equation}
    C^{k,\alpha}(\Omega)
    \coloneqq 
    \left\{
        u \in C^{k}(\Omega) \;:\;
        [D^\beta u]_{C^{0,\alpha}(\Omega)} < \infty 
        \text{ for all multi-indices } \beta \text{ with } |\beta| = k
    \right\}.
\end{equation}

We also define the following Sobolev spaces
\begin{align}
    H^s_{\overline{\Omega}}(\R^n) = \{u\in H^s(\R^n), \text{supp } u\subset \overline{\Omega}\}, \quad
    H^s(\Omega) = \{ u|_{\Omega} : u \in H^s(\R^n) \}.
\end{align}
\cle{If $\Omega$ is Lipschitz, then the space
  $H^s_{\overline{\Omega}}(\R^n)$ is the closure of
  $\C^{\infty}_c(\Omega)$ in the $H^s$-norm, see~\cite[Theorem 3.29,
    3.30]{McLean2000}.} The space $H^s(\Omega)$ is identified with the
  quotient $H^s(\Omega) = H^s(\R^n) / H^s_{\overline{\R^n \setminus
      \Omega}}(\R^n)$, with the norm $\norm{u}_{H^s(\Omega)} = \inf_{v
    \in H^s(\R^n), v|_{\Omega} = u} \norm{v}_{H^s(\R^n)}$.  The spaces
  $H^{-s}_{\overline{\Omega}}(\R^n)$ and $H^s(\Omega)$ are canonically
  dual with respect to the duality pairing \cle{defined} by \cle{extending the $L^2$} inner-product
  $\inprod{\cdot}{\cdot}_{L^2(\R^n)} : C_c^{\infty}(\Omega) \times
  C_c^{\infty}(\Omega) \to \C.$ \cle{This is possible because the
    Cauchy Schwarz inequality implies that for any $s\in\R$ we have
    the estimate
\begin{equation}\label{CS:gen}
  |(u,v)_{L^2}|\leq\|u\|_{H^{s}}\|v\|_{H^{-s}},
\end{equation}
 and $C^{\infty}_c(\Omega)\subset H^s_{\overline{\Omega}}(\R^n)$ is
 dense for any $s.$ If $v_1=v_2$ in $\Omega,$ then for $\varphi\in
 C^{\infty}_c(\Omega)$ $(\varphi,v_1)=(\varphi,v_2).$ Using these
 observations and the estimate in~\eqref{CS:gen} we easily establish
 the claim that the extension of the $L^2$-pairing defines an
 isomorphism:
 \begin{equation}
   H^{-s}_{\overline{\Omega}}(\R^n)\simeq [H^s(\Omega)]'.
 \end{equation}
 We denote this extended pairing by $\inprod{u}{v}_{L^2(\Omega)}:H^s(\Omega)\times H^{-s}_{\overline{\Omega}}(\R^n),$ for all $s.$
}

Functions in $H^s_{\overline{\Omega}}(\mathbb{R}^n)$ are defined on
$\mathbb{R}^n$ but have compact support contained in
$\overline{\Omega}$. Therefore, for any larger domain $O \supset
\Omega$, the space $H^s_{\overline{\Omega}}(\mathbb{R}^n)$ is
equivalent to $H^s_{\overline{\Omega}}(O)$. For any $s,$ the operation
of extension by zero is a bounded map from
$H^s_{\overline{\Omega}}(\R^n)\to H^s(\R^n).$
\begin{lemma}
    \label{lem:hso-norm}
    For any non-empty open sets $\Omega,O\subset\R^n$ where $\overline{\Omega}\subset O$, it follows that $H^s_{\overline{\Omega}}(\R^n)$ is equal to $H^s_{\overline{\Omega}}(O)$.
\end{lemma}
\begin{proof}
See Appendix~\ref{pf-hso-norm}. 
\end{proof}

We also introduce functional spaces on the boundary $\Gamma=\partial\Omega$.
For $s \in \mathbb N$, we define the Sobolev space on $\Gamma$ as
\begin{equation}
    H^s(\Gamma) = \{u \in L^2(\Gamma): X^\alpha u \in L^2(\Gamma) \text{ for } |\alpha| \le s \},
\end{equation}
where $X^\alpha$ is the \cle{unit length} vector field tangent to $\Gamma$.
For $s \in \R_+\setminus\mathbb N^*$, the fractional Sobolev spaces are defined by interpolation\rev{ \cite[Appendix~B]{McLean2000}}, $H^{k+\theta}(\Gamma)=[H^{k}(\Gamma),H^{k+1}(\Gamma)]_{\theta}$ with $s = k+\theta$ and $k\in\mathbb N$. They can also be characterized with the trace operator
\begin{equation}
    H^s(\Gamma) =\{ u|_{\Gamma} : u \in H^{s+\frac{1}{2}}(\R^2) \}, \quad s > 0.
\end{equation}
The negative Sobolev space for $s < 0$ is defined via duality
\begin{equation}
        H^s(\Gamma) = \left(H^{-s}(\Gamma)\right)', \quad s < 0.
\end{equation}
Similarly, for the open portion $S \subset \Gamma$, we define the Sobolev spaces
\begin{align}
    H^s_{\overline{S}}(\Gamma) = \{ u \in H^s(\Gamma) : \operatorname{supp}\, u \subset \overline{S} \} , \quad
    H^s(S) = \{ u|_{S} : u \in H^s(\Gamma) \} .
\end{align}
On the boundary $\partial S$, which consists of a finite number of points, we define the inner-product and norm as
\begin{equation}
    \inprod{f}{g}_{\partial S} = \sum_{p\in\partial S}f(p) \overline{g(p)}, \quad \norm{f}_{\partial S} = \sqrt{\inprod{f}{f}_{\partial S}}.
\end{equation}

We denote by $\mathbb H = H^1(S) \times H^{-\frac{1}{2}}_{\overline{S}}(\Gamma).$ For any $\bm u = (u_1,u_2),\bm v = (v_1,v_2)\in \mathbb H$, the inner-product and induced norm are naturally given by
\begin{equation}
    \inprod{\bm u}{\bm v}_{\mathbb H} = \inprod{u_1}{v_1}_{H^1(S)} + \inprod{u_2}{v_2}_{H^{-\frac{1}{2}}_{\overline{S}}(\Gamma)}, \quad\norm{\bm u }_{\mathbb H} = \sqrt{\inprod{\bm u}{\bm u}_{\mathbb H}}.
\end{equation}


For any integer $k\ge 1$ and $\alpha\in(0,1]$, let $\bm{X} \in C^{k,\alpha}(I)$ where $I=(-1,1)$ be a \emph{regular} parameterization of $S$. 
Here, we say the parametrization $\bm X$ is regular if $|\partial_s \bm{X}(s)| > 0$ for all $s \in I$. 
We assume that $S$ is not self-intersecting and satisfies the arc–chord condition 
\begin{equation}
    \starnorm{\bm{X}} \coloneqq \inf_{\substack{s,\eta \in I , s \neq \eta}} 
    \frac{|\bm{X}(s) - \bm{X}(\eta)|}{|s - \eta|} > 0.
\end{equation}
Furthermore, we assume that $\bm{X}$ admits an extension 
$\wt{\bm{X}}\in C^{k,\alpha}(\T)$, which parametrizes the closed curve $\Gamma$ and satisfies $|\partial_s \wt{\bm{X}}| > 0$ and $|\wt{\bm{X}}|_* > 0$ where $\mathbb{T} = \mathbb{R} / (2\pi \mathbb{Z}) $ is the one-dimensional circle.
The following norm equivalence result will be useful. 
\begin{lemma}\label{lem:norm-equv-1}
    Let $f\in H^s(\Gamma)$ for $s\in \R$ and $\bm{X}\in C^{k,\alpha}(\T)$ with $|s|\le k$ be a regular parameterization of $\Gamma$. We have
    \begin{align}
        c\norm{f \circ \bm{X}}_{H^s(\mathbb T)} \le \norm{f}_{H^s(\Gamma)}  \le  C\norm{f \circ \bm{X}}_{H^s(\mathbb T)} ,
    \end{align}
    where the constants $c,C>0$ are independent of $f$.
\end{lemma}
\begin{proof}
See Appendix~\ref{sec:pf-lem-1}.
\end{proof}
The above result immediately leads to the following norm equivalence for functions in  $H^s(S)$ and $H_{\overline{S}}^s(\Gamma)$.
\begin{lemma}\label{lem:norm-equv-2}
    Let $f\in H^s(S)$ and $g\in H^s_{\overline{S}}(\Gamma)$ for $s\in \R$ and $\bm{X}\in C^{k,\alpha}(I)$ with $|s|\le k$ be a regular parameterization of $S$. We have
    \begin{align}
        c\norm{f \circ \bm{X}}_{H^s(I)} \le \norm{f}_{H^s(S)} \le C\norm{f \circ \bm{X}}_{H^s(I)},\\
        c \norm{g \circ \bm{X}}_{H^s_{\overline{I}}(\mathbb T)} \le \norm{g}_{H^s_{\overline{S}}(\Gamma)} \le C \norm{g \circ \bm{X}}_{H^s_{\overline{I}}(\mathbb T)} ,
    \end{align}
    where the constants $c,C>0$ are independent of $f$ and $g$.
\end{lemma}

\section{Existence and uniqueness}\label{sec:exist}
To establish the well-posedness of~\eqref{eqn:arc-pde}, we begin by reformulating the problem as an integro-differential equation on $S$ via boundary integral operators.
\rev{We represent the solution $u$ as a single-layer potential, see, for example, \cite[Chapters~6--7]{McLean2000} and \cite[Chapter~2]{Hsiao2021},}
\begin{equation}\label{eqn:single-layer-rep}
    u(x) = V_\Omega \psi (x),\quad x \in \R^2, 
\end{equation}
where the single-layer potential $V_\Omega \psi:H^{-\frac{1}{2}}(\Gamma)\to H^1_{loc}(\R^2)$ is defined as
\begin{equation}\label{eqn:single-layer-potential}
    V_\Omega \psi (x) = -\frac{1}{2\pi}\int_\Gamma \log|x - y| \psi(y) \,ds_y,\quad x \in \R^2.
\end{equation}
\cle{This function is harmonic in $\R^2\setminus\Gamma.$} \cle{ If
$\text{supp }\psi\subset \overline{S},$ then $\psi =
\jump{\partial_{\bm n} u} \in H^{-\frac{1}{2}}_{\overline{S}}(\Gamma),$}
\rev{see~\cite[Theorem~6.11]{McLean2000}}.
Let $V_\Gamma \psi$ be the trace of $V_{\Omega}\psi$ on $\Gamma$ and
$V_S\psi = p_S V_\Gamma\psi$ be the restriction on $S$.  We now
consider a slightly more general integro-differential system given by
\begin{subequations}\label{eqn:arc-bie}
    \begin{align}
        -\Delta_\Gamma U + \psi &= f,\quad \text{ on } S, \label{eqn:arc-bie-b}\\
        V_S \psi - U &= h, \quad \text{ on } S, \\
        \bm \nu\cdot\grad_\Gamma U &= g,\quad \text{ on } \partial S.
    \end{align}
\end{subequations}
The case when $h=0$ corresponds to the problem~\eqref{eqn:arc-pde}.



We first state a result concerning the uniqueness of the solution to problem~\eqref{eqn:arc-pde}.
\begin{proposition}\label{unique-pde}
    The pair of functions $ (u, U) \in
    H^1_{\mathrm{loc}}(\mathbb{R}^2) \times H^1(S) $
    solves~\eqref{eqn:arc-pde} \cle{with $f=g=0$} if and only if $ u
    \equiv 0 $ and $ U \equiv 0 $.
\end{proposition}
\begin{proof}
Let $B_r$ be a ball of radius $r$ centered at the origin, with $r$ large enough such that $\Gamma \subset B_r$. Since $u \in H^1_{\mathrm{loc}}(\mathbb{R}^2)$, its traces satisfy $\at{u}{\Gamma}\in H^{\frac12}(\Gamma) $ and $\at{u}{\p B_r}\in H^{\frac12}(\p B_r) $. Let $u_1 = p_{\Omega_1} u$ and $u_2 = p_{B_r\setminus\overline{\Omega_1}} u$. By \cite[Lemma 7.1]{Epstein2010Debye} and \cite{taylor1996pde1}, the normal derivatives satisfy $\at{\p_{\bm n}u_i}{\Gamma}\in H^{-\frac12}(\Gamma)$ and $\at{u_2}{\p B_r}\in H^{\frac12}(\p B_r)$.
The jump in the normal derivative across $\Gamma$ is given by $\jump{\p_{\bm n}u} = \at{\p_{\bm n}u_1}{\Gamma} - \at{\p_{\bm n}u_2}{\Gamma} \in H^{-\frac12}(\Gamma)$. 
Since $u$ satisfies the Laplace equation~\eqref{eqn:arc-pde-a} in $\R^2\setminus\overline{S}$, it is smooth away from $\overline{S}$, which implies $\jump{u}= 0$ and $\jump{\p_{\bm n}u} = 0$ on $\Gamma\setminus\overline{S}$.
Thus, $\jump{\p_{\bm n}u} \in H^{-\frac12}_{\overline{S}}(\Gamma)$.
Multiplying both sides of~\eqref{eqn:arc-pde-a} by $u$, integrating over $B_r$, and applying the integration by parts formula, we obtain
\begin{equation}
\begin{aligned}
    \int_{B_r} |\nabla u|^2 \, d\bm{x} &= \int_{\partial B_r} u\, \partial_{\bm{n}} u \, ds + \int_\Gamma (u_1\p_{\bm n}u_1 - u_2\p_{\bm n}u_2) \, ds, \\
    &= \int_{\partial B_r} u\, \partial_{\bm{n}} u \, ds + \int_S U \jump{\p_{\bm n}u} \, ds.
\end{aligned}
\end{equation}
Using the homogeneous equations~\eqref{eqn:arc-pde-b}--\eqref{eqn:arc-pde-d} (with $f=g=0$) and integrating by parts on $S$, we have
\begin{equation}
\begin{aligned}
    \int_{B_r} |\nabla u|^2 \, d\bm{x} &= \int_{\partial B_r} u\, \partial_{\bm{n}} u \, ds + \int_S U \Delta_\Gamma U \, ds  = \int_{\partial B_r} u\, \partial_{\bm{n}} u \, ds - \int_S |\nabla_\Gamma U|^2 \, ds.
\end{aligned}
\end{equation}
Since $u$ is harmonic in $\R^2\setminus\overline{B_r}$ and satisfies the decay condition~\eqref{eqn:arc-pde-e}, it follows that $\abs{\grad u} = \mc O(r^{-2})$ by \cite[Theorem 8.8]{McLean2000}. Consequently, we have $\int_{\p B_r}u\p_{\bm n}u\,ds \to 0$ as $r\to\infty$. 
Taking the limit $r \to \infty$, we conclude that both $ \nabla u $ and $ \nabla_\Gamma U $ must vanish, implying that $ u $ and $ U $ are constant functions.
The decay condition $u = o(1)$ as $ |\bm{x}| \to \infty $ then forces $ u \equiv 0 $ and $ U \equiv 0 $ everywhere.
\end{proof}
We now show the correspondence between solutions of \eqref{eqn:arc-pde} and \eqref{eqn:arc-bie}.
\begin{proposition}
    If $(U, \psi)$ is a solution to~\eqref{eqn:arc-bie} with $h=0$, then $(V_\Omega \psi, U)$ is the unique solution to~\eqref{eqn:arc-pde}.
\end{proposition}
\begin{proof}
Suppose $ (U, \psi) $ is a solution to \eqref{eqn:arc-bie}. By integrating both sides of \eqref{eqn:arc-bie-b} and applying integration by parts, we obtain
\begin{equation}\label{eqn:int_psi}
    \int_S \psi(x) \,ds_x = \int_S f(x) \,ds_x + \sum_{p\in\partial S}g(p) = 0,
\end{equation}
which implies that the decay \cle{condition} at infinity, \eqref{eqn:arc-pde-e}, is automatically satisfied as long as $f$ and $g$ satisfy the compatibility condition~\eqref{eqn:compatible}.
Since $ V_\Omega \psi $ satisfies the Laplace equation together with the boundary conditions on $ S $, it follows that $ (V_\Omega \psi, U) $ is a solution to problem~\eqref{eqn:arc-pde}. 
Combining this with the uniqueness result in Proposition~\ref{unique-pde}, we complete the proof.
\end{proof}

To establish the existence and regularity of solutions to~\eqref{eqn:arc-bie}, we start with a weak formulation and \cle{prove the existence} of  weak solutions. In this setting, we assume that the curve is of class $C^{1,\delta}$.
Here, $h$ is allowed to be nonzero.
Given $(f,h)\in \mathbb H'\cle{=H^{-1}_{\overline{S}}(\Gamma)\times H^{\frac 12}(S)}$ and $g$ with $\norm{g}_{\partial S}<\infty$, we say $\bm u = (U, \psi)\in\mathbb H$ is a weak solution to~\eqref{eqn:arc-bie} if it satisfies
\begin{equation}\label{eqn:weak-eqn}
  a(\bm u, \bm v) = b(\bm v), \quad \forall \bm v = (\phi, \zeta) \in \mathbb H=
  \cle{ H^1(S)\times H^{-\frac 12}_{\overline{S}}(\Gamma)},
\end{equation}
where the bilinear form $a(\cdot,\cdot)$ and the linear form $b(\cdot)$ are given by
\begin{align}
    a(\bm u, \bm v) &= \inprod{\grad_\Gamma U}{\grad_\Gamma \phi}_{L^2(S)} +  \dual{\psi}{\phi}_{L^2(S)} + \dual{V_S\psi }{\zeta}_{L^2(S)} - \dual{ U}{\zeta}_{L^2(S)}, \\
    b(\bm v) & = \dual{f}{\phi}_{L^2(S)} + \dual{h}{\zeta}_{L^2(S)}+ \inprod{g}{\phi}_{\partial S}.
\end{align}

\begin{lemma}\label{lem:ab-conti}
    The bilinear form $a(\cdot, \cdot)$ is continuous and the linear form $b(\cdot)$ is bounded on $\mathbb H$.
\end{lemma}
\begin{proof}
    By definition, $\dual{V_S\psi}{\zeta}_{L^2(S)} = \dual{V_\Gamma\psi}{\zeta}_{L^2(\Gamma)}$ for every $\zeta,\psi\in H^{-\frac{1}{2}}_{\overline{S}}(\Gamma)$ and
    \begin{equation}
        \norm{V_S\psi}_{H^{\frac{1}{2}}(S)} \le \norm{V_\Gamma \psi}_{H^\frac{1}{2}(\Gamma)} \le C \norm{\psi}_{H^{-\frac{1}{2}}(\Gamma)} = C\norm{\psi}_{H^{-\frac{1}{2}}_{\overline{S}}(\Gamma)}, \quad \forall \psi\in H^{-\frac{1}{2}}_{\overline{S}}(\Gamma).
    \end{equation}
    where we have used that the single-layer integral on a closed curve $V_\Gamma:H^{-\frac{1}{2}}(\Gamma)\to H^{\frac{1}{2}}(\Gamma)$ is bounded \cite[Theorem 7.1]{McLean2000}.
    \rev{The following estimate uses Cauchy--Schwarz in the duality pairings and the continuous embeddings $H^1(S)\hookrightarrow H^{\frac12}(S)$ and $H^{-\frac12}_{\overline S}(\Gamma)\hookrightarrow H^{-1}_{\overline S}(\Gamma)$.}
    For every $\bm u, \bm v\in\mathbb H$,
    \begin{equation}
    \begin{aligned}
        |a(\bm u, \bm v)| &\le C (\norm{\grad_\Gamma U}_{L^2(S)} \norm{\grad_\Gamma\phi}_{L^2(S)} + \norm{\psi}_{H^{-\frac{1}{2}}_{\overline{S}}(\Gamma)} \norm{\phi}_{H^\frac{1}{2}(S)}  \\
        &\quad+  \norm{V_S\psi}_{H^\frac{1}{2}(S)} \norm{\zeta}_{H^{-\frac{1}{2}}_{\overline{S}}(\Gamma)} + \norm{U}_{H^\frac{1}{2}(S)} \norm{\zeta}_{H^{-\frac{1}{2}}_{\overline{S}}(\Gamma)} ) \\
        &\le C ( \norm{U}_{H^1(S)} \norm{\phi}_{H^1(S)}  +  \norm{\psi}_{H^{-\frac{1}{2}}_{\overline{S}}(\Gamma)} \norm{\phi}_{H^1(S)} \\
        &\quad + \norm{\psi}_{H^{-\frac{1}{2}}_{\overline{S}}(\Gamma)} \norm{\zeta}_{H^{-\frac{1}{2}}_{\overline{S}}(\Gamma)} + \norm{U}_{H^1(S)} \norm{\zeta}_{H^{-\frac{1}{2}}_{\overline{S}}(\Gamma)} ) = C \norm{\bm u}_{\mathbb H} \norm{\bm v}_{\mathbb H}.
    \end{aligned} 
    \end{equation}
    For every $\bm v = (\phi, \zeta)\in\mathbb H$, \rev{Cauchy--Schwarz and the Sobolev embedding $H^1(S)\hookrightarrow C^{0,\frac{1}{2}}(S)$ for the endpoint terms lead to the following estimate,}
    \begin{equation}
    \begin{aligned}
        |b(\bm v)| &\le \norm{f}_{H^{-1}_{\overline{S}}(\Gamma)} \norm{\phi}_{H^1(S)} +  \norm{h}_{H^\frac{1}{2}(S)}\norm{\zeta}_{H^{-\frac{1}{2}}_{\overline{S}}(\Gamma)} + (\sum_{p\in \partial S} |g(p)|^2)^{\frac{1}{2}} (\sum_{p\in \partial S}|\phi(p)|^2)^{\frac{1}{2}} \\
        &\le \norm{f}_{H^{-1}_{\overline{S}}(\Gamma)} \norm{\phi}_{H^1(S)} +\norm{h}_{H^\frac{1}{2}(S)}\norm{\zeta}_{H^{-\frac{1}{2}}_{\overline{S}}(\Gamma)}+ C \norm{g}_{\partial S} \norm{\phi}_{H^1(S)} \\
        &\le C\paren{\norm{f}_{H^{-1}_{\overline{S}}(\Gamma)} +\norm{h}_{H^\frac{1}{2}(S)}+ \norm{g}_{\partial S}} \norm{\bm v}_{\mathbb H}.
    \end{aligned}
    \end{equation}
    This completes the proof.
\end{proof}

\begin{lemma}\label{lem:vg-decomp}
    Let $\Gamma$ be of class $C^{1,\delta}$ with $\delta\in (0,1)$ and non-self-intersecting. The operator $V_\Gamma:H^{-\frac{1}{2}}(\Gamma)\to H^{\frac{1}{2}}(\Gamma)$ can be decomposed into $V_\Gamma = V_{\Gamma, 0} + V_{\Gamma, 1}$ where $V_{\Gamma,0}:H^{-\frac{1}{2}}(\Gamma)\to H^{\frac{1}{2}}(\Gamma)$ is positive and bounded below, i.e., there exists $C>0$ such that
    \begin{equation}
        \dual{V_{\Gamma, 0}\psi}{\psi}_{L^2(\Gamma)} \ge C \norm{\psi}_{H^{-\frac{1}{2}}(\Gamma)}^2, \quad \forall \psi\in H^{-\frac{1}{2}}(\Gamma),
    \end{equation}
    and $V_{\Gamma,1}:H^{-1+s}(\Gamma)\to H^{s+\varepsilon}(\Gamma)$ is bounded for every $s\in[0,1]$ and $\varepsilon < \delta$.
\end{lemma}
\begin{proof}
The proof is technical and is provided in the Appendix~\ref{pf-vg-decom}.
\end{proof}
The above result can be used directly to obtain a decomposition for $V_S$.
\begin{lemma}\label{lem:vs-decomp}
    Let $S$ be of class $C^{1,\delta}$ with $\delta\in (0,1)$ and non-self-intersecting. The operator $V_S:H^{-\frac{1}{2}}_{\overline{S}}(\Gamma)\to H^{\frac{1}{2}}(S)$ can be decomposed into $V_S = V_{S, 0} + V_{S, 1}$ where $V_{S,0}:H^{-\frac{1}{2}}_{\overline{S}}(\Gamma)\to H^{\frac{1}{2}}(S)$ is positive and bounded below, i.e., there exists $C>0$ such that
    \begin{equation}
        \dual{V_{S, 0}\psi}{\psi}_{L^2(S)} \ge C \norm{\psi}_{H^{-\frac{1}{2}}_{\overline{S}}(\Gamma)}^2, \quad \forall \psi\in H^{-\frac{1}{2}}_{\overline{S}}(\Gamma),
    \end{equation}
    and $V_{S,1}:H^{-1+s}_{\overline{S}}\rev{(\Gamma)}\to H^{s+\varepsilon}(S)$ is bounded for every $s\in[0,1]$ and $\varepsilon < \delta$.
\end{lemma}
\begin{proof}
    Since $H^{-\frac{1}{2}}_{\overline{S}}(\Gamma)$ is a closed subspace of $H^{-\frac{1}{2}}(\Gamma)$, we can define $V_{S,i}$ as restrictions of $V_{\Gamma,i}$ for $i=0,1$ such that,
    \begin{equation}
        V_{S,i}\psi = \rev{p_S} V_{\Gamma,i}\psi,        \quad \forall\psi \rev{\in} H^{-\frac{1}{2}}_{\overline{S}}(\Gamma).
    \end{equation}
    Then it follows that for every $\psi\in H^{-\frac{1}{2}}_{\overline{S}}(\Gamma)$,
    \begin{equation}
        \dual{V_{S,0}\psi}{\psi}_{L^2(S)} = \dual{V_{\Gamma,0}\psi}{\psi}_{L^2(\Gamma)} \ge C \norm{\psi}_{H^{-\frac{1}{2}}(\Gamma)}^2 = C\norm{\psi}_{H^{-\frac{1}{2}}_{\overline{S}}(\Gamma)}^2,
    \end{equation}
    \rev{Moreover, for every $s\in[0,1]$, Lemma~\ref{lem:vg-decomp} gives the closed-curve estimate}
    \begin{equation}
        \rev{\norm{V_{\Gamma,1}\psi}_{H^{s+\varepsilon}(\Gamma)}
        \le C\norm{\psi}_{H^{-1+s}(\Gamma)}
        = C\norm{\psi}_{H^{-1+s}_{\overline S}(\Gamma)}.}
    \end{equation}
    \rev{Restriction from $\Gamma$ to $S$ leads to}
    \begin{equation}
        \rev{\norm{V_{S,1}\psi}_{H^{s+\varepsilon}(S)}
        \le C\norm{\psi}_{H^{-1+s}_{\overline S}(\Gamma)}.}
    \end{equation}
    \rev{Taking $s=\frac12$ shows that $V_{S,1}$ maps $H^{-\frac12}_{\overline S}(\Gamma)$ boundedly into $H^{\frac12+\varepsilon}(S)$. Since the embedding $H^{\frac{1}{2}+\varepsilon}(S) \hookrightarrow H^{\frac{1}{2}}(S)$ is compact for $\varepsilon > 0$, we can also conclude that the operator $V_{S,1}:H^{-\frac{1}{2}}_{\overline{S}}(\Gamma)\to H^{\frac{1}{2}}(S)$ is compact.}
\end{proof}

In view of the proof of Lemma~\ref{lem:vg-decomp}, we also have
\begin{lemma}\label{lem:vs-reg}
    Let $S$ be of class $C^{1,\delta}$ with $\delta\in (0,1)$ and non-self-intersecting. The operator $V_S:H^{-\frac{1}{2}+\varepsilon}_{\overline{S}}(\Gamma)\to H^{\frac{1}{2}+\varepsilon}(S)$ is bounded for every $|\varepsilon| < \min(\frac{1}{2},\delta)$. 
\end{lemma}
\begin{proof}
\rev{Since \(V_S\psi=p_SV_\Gamma\psi\), it suffices to estimate \(V_\Gamma\psi\). In the proof of Lemma~\ref{lem:vg-decomp}, \(V_{\Gamma,0}=\mc L_1+\mc R\), where \(\mc L_1\) is the order \(-1\) Fourier multiplier with symbol \(1/(2|k|)\), \(k\ne0\), and \(\mc R\) is smoothing. Hence}
\[
    \rev{\norm{V_{\Gamma,0}\psi}_{H^{\frac12+\varepsilon}(\Gamma)}
    \le C\norm{\psi}_{H^{-\frac12+\varepsilon}(\Gamma)}.}
\]
\rev{For \(V_{\Gamma,1}\), taking \(s=\frac12+\varepsilon\in(0,1)\) in Lemma~\ref{lem:vg-decomp} gives, for \(0<\eta<\delta\),}
\[
    \rev{\norm{V_{\Gamma,1}\psi}_{H^{\frac12+\varepsilon+\eta}(\Gamma)}
    \le C\norm{\psi}_{H^{-\frac12+\varepsilon}(\Gamma)}.}
\]
\rev{The embedding \(H^{\frac12+\varepsilon+\eta}(\Gamma)\hookrightarrow H^{\frac12+\varepsilon}(\Gamma)\), together with restriction from \(\Gamma\) to \(S\), yields}
\[
    \rev{\norm{V_S\psi}_{H^{\frac12+\varepsilon}(S)}
    \le C\norm{\psi}_{H^{-\frac12+\varepsilon}_{\overline S}(\Gamma)},}
\]
\rev{which proves the lemma.}
\end{proof}

Since $a(\cdot,\cdot)$ is continuous in $\mathbb H$, we can define a bounded linear operator $\mathcal{A}:\mathbb H\to \mathbb H'$ by
\begin{equation}
    \dual{\mc A \bm u}{\bm v}_{L^2(S)} = a(\bm u, \bm v),\quad \forall \bm u, \bm v\in\mathbb H.
\end{equation}
It can be shown that the operator $\mc A$ satisfies a G\aa rding inequality.
\begin{theorem}\label{thm:solvability}
    Let $S$ be of class $C^{1,\delta}$ with $\delta\in (0,1)$. The following results hold:
    
    (i) There exists a compact operator $\mc B:\mathbb H\to \mathbb H'$ and a constant $C>0$ such that the G\aa rding inequality holds
    \begin{equation}\label{eqn:garding}
        \dual{(\mc A + \mc B) \bm u}{\bm u}_{L^2(S)} \ge C \norm{\bm u}_{\mathbb H}^2, \quad \forall \bm u\in \mathbb H.
    \end{equation}
    (ii) Given $f\in H^{-1}_{\overline{S}}(\Gamma)$ and $g$ with $\norm{g}_{\partial S}<\infty$, there exists a unique weak solution $(U, \psi)\in\mathbb H$ to \eqref{eqn:weak-eqn} with
    \begin{equation}
        \norm{U}_{H^1(S)} + \norm{\psi}_{H^{-\frac{1}{2}}_{\overline{S}}(\Gamma)} \le C\paren{ \norm{f}_{H^{-1}_{\overline{S}}(\Gamma)} +\norm{h}_{H^\frac{1}{2}(S)}+ \norm{g}_{\partial S}}.
    \end{equation}   
\end{theorem}
\begin{proof}
    Define the operator $\mc B $ as
    \begin{equation}\label{eqn:def-B}
    \dual{\mc B\bm u}{\bm v}_{L^2(S)} = \inprod{U}{\phi}_{L^2(S)} - \dual{V_S\psi}{\zeta}_{L^2(S)} , \quad \forall \bm u, \bm v \in L^2(S) \times H^{-\frac{1}{2}}_{\overline{S}}(\Gamma).
    \end{equation}
    
    For any $\varepsilon\in(0,\frac{\delta}{2})$, by Lemma~\ref{lem:vs-decomp}, we have,  
    \begin{equation}
    \begin{aligned}
        &\norm{\mc B \bm u}_{H^{-1+\varepsilon}_{\overline{S}}(\Gamma)\times H^{\frac{1}{2}+\varepsilon}(S)} \\
        &= \sup_{(\phi,\zeta)\neq 0} \frac{|\inprod{U}{\phi}_{L^2(S)} - \dual{V_{\rev{S}}\psi}{\zeta}_{L^2(S)}|}{\norm{\phi}_{H^{1-\varepsilon}(S)} + \norm{\zeta}_{H^{-\frac{1}{2}-\varepsilon}_{\overline{S}}(\Gamma)}} \\
        &\le \sup_{(\phi,\zeta)\neq 0}\frac{\norm{U}_{H^{-1+\varepsilon}_{\overline{S}}(\Gamma)} \norm{\phi}_{H^{1-\varepsilon}(S)} + \norm{V_{\rev{S}}\psi}_{H^{\frac{1}{2}+\varepsilon}(S)} \norm{\zeta}_{H^{-\frac{1}{2}\rev{-}\varepsilon}_{\overline{S}}(\Gamma)}}{\norm{\phi}_{H^{1-\varepsilon}(S)} + \norm{\zeta}_{H^{-\frac{1}{2}-\varepsilon}_{\overline{S}}(\Gamma)}} \\
        &\le \norm{U}_{H^{-1+\varepsilon}_{\overline{S}}(\Gamma)} + \norm{V_S\psi}_{H^{\frac{1}{2}+\varepsilon}(S)} \le \rev{C}\norm{U}_{L^2(S)} + \rev{C}\norm{\psi}_{H^{-\frac{1}{2}-\varepsilon}_{\overline{S}}(\Gamma)}\\ 
        &\le \rev{C}\norm{U}_{H^1 (S)} + \rev{C}\norm{\psi}_{H^{-\frac{1}{2}}_{\overline{S}}(\Gamma)} =  \rev{C}\norm{\bm u}_{\mathbb H}.        
    \end{aligned}
    \end{equation}
    This implies $\mc B:\mathbb H\to \mathbb H'$ is compact since $H^s_{\overline{S}}(\Gamma)\hookrightarrow \hookrightarrow H^t_{\overline{S}}(\Gamma)$ for $s > t$.
    The G\aa rding inequality~\eqref{eqn:garding} follows directly,
    \begin{equation}
    \begin{aligned}
        \dual{\mc A\bm u}{\bm u}_{L^2(S)} &= \norm{\grad_\Gamma U}_{L^2(S)}^2 + \dual{V_S\psi}{\psi}_{L^2(S)} \\
        & \ge \norm{U}_{H^1(S)}^2 - \norm{U}_{L^2(S)}^2+ C \norm{\psi}_{H^{-\frac{1}{2}}_{\overline{S}}(\Gamma)}^2  + \dual{V_{S, 1}\psi}{\psi}_{L^2(S)} \\
        & \ge C \norm{\bm u}_{\mathbb H}\rev{^2} - \dual{\mc B\bm u}{\bm u}_{L^2(S)}, \quad \forall \bm u = (U, \psi)\in\mathbb H.
    \end{aligned}
    \end{equation}
    Hence, $\mc A + \mc B:\mathbb H\to \mathbb H'$ is strictly coercive and $\mc A$ is a Fredholm operator with index zero.
    Therefore, the solvability of the problem~\eqref{eqn:weak-eqn} is equivalent to the uniqueness.
    In order to show uniqueness, we consider the homogeneous problem for $\bm u\in\mathbb H$,
    \begin{equation}
        a(\bm u, \bm v) = 0, \quad \bm v\in\mathbb H.
    \end{equation}
    Taking $\bm v = (1, 0)$, we obtain
    \begin{equation}
        \dual{1}{\psi}_{L^2(S)} = \dual{1}{\psi}_{L^2(\Gamma)} = 0.
    \end{equation}
   Taking $\bm v = \bm u$, we obtain
    \begin{equation}
        \norm{\grad_\Gamma U}_{L^2(S)}^2 + \dual{V_S \psi}{\psi}_{L^2(S)} = \norm{\grad_\Gamma U}_{L^2(S)}^2 + \dual{V_\Gamma \psi}{\psi}_{L^2(\Gamma)} = 0.
    \end{equation}
    Define $H^{-\frac{1}{2}}_0(\Gamma)$ as
    \begin{equation}
         H^{-\frac{1}{2}}_0(\Gamma) = \{\psi\in H^{-\frac{1}{2}}(\Gamma): \dual{1}{\psi}_{L^2(\Gamma)} = 0\}.
    \end{equation}
    By \cite[Theorem 8.12]{McLean2000}, $V_\Gamma$ is strictly positive-definite on $H^{-\frac{1}{2}}_0(\Gamma)$, that is,
    \begin{equation}
        \dual{V_\Gamma \psi}{\psi}_{L^2(\Gamma)} > 0 \quad \forall \psi \in H^{-\frac{1}{2}}_0(\Gamma)\setminus\{0\}.
    \end{equation}
    So we have $\psi = 0$. Let $\phi = 0$. We have $\dual{U}{\zeta} = 0, \forall \zeta\in H^{-\frac{1}{2}}_{\overline{S}}(\Gamma)$ and so $U = 0$. By the Fredholm alternative, $\mc A$ is bijective and has bounded inverse. By Lemma~\ref{lem:ab-conti}, it holds that
    \begin{equation}
    \begin{aligned}
        \norm{\bm u}_{\mathbb H} &\le C \norm{\mc A\bm u}_{\mathbb H'}  = C \sup_{0\neq\bm v\in\mathbb H}\frac{a(\bm u,\bm v)}{\norm{\bm v}_{\mathbb H}} = C \sup_{0\neq\bm v\in\mathbb H}\frac{b(\bm v)}{\norm{\bm v}_{\mathbb H}} \\
        &\le C \paren{\norm{f}_{H^{-1}_{\overline{S}}(\Gamma)} +\norm{h}_{H^\frac{1}{2}(S)}+ \norm{g}_{\partial S}}.
    \end{aligned}
    \end{equation}
    This completes the proof.
\end{proof}

\section{Regularity of solutions}\label{sec:reg}


\subsection{Half-line case}
We discuss the case in which the boundary is given by $S = \R_+ \times \{0\}$, where $\R_+ = (0, +\infty)$ is the half-line.
Let $A:H^{s}(\R)\to H^{s+1}(\R)$ and $B:H^{s}(\R)\to H^{s-2}(\R)$ be pseudo-differential operators with symbols $\wh A(\xi) = (1 + \xi^2)^{-\frac{1}{2}}$ and $\wh B(\xi) = 1 + \xi^2$, respectively. In fact, $B$ is a differential operator $   \mathcal{I}- \partial_{ss}$.
We consider the pseudo-differential equations on the half-line,
\begin{equation}\label{eqn:pdo-A}
    p_+ A \psi_+ = f_1, \quad \text{ on }\R_+,
\end{equation}
and,
\begin{equation}\label{eqn:pdo-B}
    B v = f_2, \quad \text{ on }\R_+, \quad \partial_x v(0) = g,
\end{equation}
where $p_+$ is the restriction to the half-line $\R_+$. Here, $\psi_+ = \psi(x)$ for $x > 0$ and $\psi_+(x) = 0$ for $x < 0$. The operator $A$ represents the leading-order behavior of the single-layer integral operator and has a suitable form for applying the Wiener--Hopf technique \cite{eskin1981boundary}.

\begin{lemma}\label{lem:A_global}
    Let $f_1 \in H^{\frac{1}{2} + \varepsilon}(\R_+)$ be given with $|\varepsilon| < \frac{1}{2}$. The solution satisfies $\psi_+ \in H^{-\frac{1}{2}+\varepsilon}_{\overline{\R_+}}(\R)$, and there exists a constant $C>0$ such that
    \begin{equation}
        \norm{\psi_+}_{H^{-\frac{1}{2}+\varepsilon}(\R)} \le C \norm{f_1}_{H^{\frac{1}{2}+\varepsilon}(\R_+) }.
    \end{equation}
\end{lemma}
\begin{proof}
Define symbols $\wh A_\pm(\xi) = (\xi \pm i)^{-\frac{1}{2}}$. Then, $\wh A_+ (\xi + i \tau)$ is analytic in the closed upper half plane where $\tau > 0$ and $\wh A_- (\xi - i \tau)$ is analytic in the closed lower half plane. We factor $\wh A (\xi)$ as
\begin{equation}
    \wh A (\xi) = (\xi + i)^{-\frac{1}{2}} (\xi - i)^{-\frac{1}{2}} = \wh A_+ (\xi) \wh A_- (\xi).
\end{equation}
Let $\wt{f}_1 \in H^{\frac{1}{2}+\varepsilon}(\R)$ be an arbitrary extension of $f_1$ satisfying
\begin{equation}\label{eqn:f1-extension-bound}
    \norm{\wt{f}_1}_{H^{\frac{1}{2}+\varepsilon}(\R)} \le 2 \norm{f_1}_{H^{\frac{1}{2}+\varepsilon}(\R_+)},
\end{equation}
and let
\begin{equation}\label{eqn:psi-minus-def}
    \psi_- = \wt{f}_1 - A \psi_+ \in H^{\frac{1}{2}+\varepsilon}_{\overline{\R_-}}(\R).
\end{equation}
\rev{Using the definition in~\eqref{eqn:psi-minus-def},~\eqref{eqn:pdo-A} can be written as}
\begin{equation}\label{eqn:wh-a-split}
    \wh A_+ (\xi) \wh A_- (\xi)\wh {\psi_+} (\xi)+ \wh{\psi_-} (\xi)= \wh{\wt{f}_1}(\xi).
\end{equation}
\rev{Multiplying~\eqref{eqn:wh-a-split} by $\wh A_- ^{-1}(\xi)$ gives}
\begin{equation}\label{eqn:apam}
    \wh A_+ (\xi)\wh {\psi_+} (\xi)+  \wh A_- ^{-1}(\xi)\wh{\psi_-}(\xi) =  \wh A_- ^{-1}(\xi)\wh{\wt{f}_1}(\xi),
\end{equation}
where each term belongs to $\mathcal{F}(H^{\varepsilon}(\R))$. 
Furthermore, by the Paley\rev{--}Wiener theorem \rev{for distributions supported on a half-line \cite[Theorem~4.5]{eskin1981boundary}}, we have
\begin{equation}
    \wh A_\pm \wh {\psi_\pm} \in  \mathcal{F}(H^{\varepsilon}_{\overline{\R_\pm}}(\R)).
\end{equation}
For any $\vph \in H^s(\R)$ with $s\in\R$, the operator $\Pi_+(\wh \vph) = \mathcal{F}(\chi_{[0,\infty)}\vph)$ can be expressed as
\begin{equation}
    \Pi_\pm (\wh \vph)(\xi) = \frac{\wh \vph(\xi)}{2} \pm \Pi_0 (\wh \vph)(\xi),\quad\Pi_0(\wh \vph)(\xi) = \frac{i}{2\pi}\text{p.v.}\int_\R  \frac{\wh \vph(\eta)d\eta}{\xi - \eta}.
\end{equation}
Since $\mathcal{F}(H^s(\R)) = L^2(\R, (1+|\xi|^2)^s d\xi)$, it can be shown that $\Pi_0: \mathcal{F}(H^s(\R))\to\mathcal{F}(H^s(\R))$ is bounded for $|s| < \frac{1}{2}$.
\rev{Hence, we have}
\begin{equation}\label{eqn:pi-plus-support}
    \Pi_+[\wh A_- \wh{\psi_-}](\xi) = 0, \quad \Pi_+[\wh A_+ \wh{\psi_+}](\xi) = \wh A_+ (\xi)\wh{\psi_+}(\xi)\in H^{\varepsilon}_{\overline{\R_+}} (\R),
\end{equation}
\rev{Using~\eqref{eqn:pi-plus-support} and applying $\Pi_+$ to~\eqref{eqn:apam}, we have}
\begin{equation}\label{eqn:a-plus-psi-pi}
    \wh A_+ (\xi)\wh{\psi_+}(\xi) = \Pi_+[\wh A_- ^{-1}  \wh{\wt{f}_1}] (\xi),
\end{equation}
\rev{and hence, from~\eqref{eqn:a-plus-psi-pi},}
\begin{equation}\label{eqn:psi-plus-fourier-form}
    \wh{\psi_+}(\xi) = \wh A_+ ^{-1}(\xi)\Pi_+[\wh A_- ^{-1}  \wh{\wt{f}_1}] (\xi) \in \mathcal{F}(H^{-\frac{1}{2}+\varepsilon}_{\overline{\R_+}}(\R)).
\end{equation}
Thus, $\psi_+ \in H^{-\frac{1}{2}+\varepsilon}_{\overline{\R_+}}(\R)$ and
\rev{the estimate follows from~\eqref{eqn:psi-plus-fourier-form}, the multiplier bounds for \(\wh A_+^{-1}\) and \(\wh A_-^{-1}\), the boundedness of \(\Pi_+\) on \(\mathcal F(H^\varepsilon(\mathbb R))\) for \(|\varepsilon|<\frac12\), and the extension bound in~\eqref{eqn:f1-extension-bound}:}
\begin{equation}
    \norm{\psi_+}_{H^{-\frac{1}{2}+\varepsilon}(\R)} \le C \norm{\wt{f}_1}_{H^{\frac{1}{2}+\varepsilon}(\R)} \le C \norm{f_1}_{H^{\frac{1}{2}+\varepsilon}(\R_+)}.
\end{equation}
\end{proof}

For more regular data, we can derive the singular behavior of the solution near the boundary based on the Wiener-Hopf technique. 
\begin{lemma}\label{lem:psi_singular}
    Let $f_1 \in H^{\frac{3}{2} + \varepsilon}(\R_+)$ be given with $|\varepsilon| < \frac{1}{2}$. The solution $\psi_+$ has the form
    \begin{equation}\label{eqn:psi-halfline-singular}
        \psi_+ (x) = c \chi_{[0,\infty)} (x) e^{-x} x^{-\frac{1}{2}} + \psi_r(x), \quad \psi_r \in H^{\frac{1}{2}+\varepsilon}_{\overline{\R_+}}(\R),
    \end{equation}
    where $c$ is a constant depending on $f_1$.
\end{lemma}
\begin{proof}
We will first consider the general case for $f_1\in H^{n + \frac{1}{2} + \varepsilon}(\R_+)$ where $n \in \mathbb N_+$.
\rev{Recall from~\eqref{eqn:psi-plus-fourier-form} that}
\begin{equation}
    \wh{\psi_+ }(\xi) = \wh A_+ ^{-1}(\xi)\Pi_+[\wh A_- ^{-1}  \wh{\wt{f}_1}] (\xi).
\end{equation}
Since $\Pi_+: \mathcal{F}(H^s_{\overline{\R_+}}(\R))\to\mathcal{F}(H^s_{\overline{\R_+}}(\R))$ is only bounded for $|s|<\frac{1}{2}$, we need to treat $\wh A_- ^{-1}  \wh{\wt{f}_1} \in \mathcal{F}(H^{n + \varepsilon}_{\overline{\R_+}}(\R)) \subset \mathcal{F}(H^{\varepsilon}_{\overline{\R_+}}(\R))$ and conclude that $\psi_+ \in H^{-\frac{1}{2}+\varepsilon}_{\overline{\R_+}}(\R)$. 
We will apply the expansion formula for the operator $\Pi_+$ \cite[Lemma 5.5]{eskin1981boundary} to extract the singular terms.
Define the symbols $\wh \Lambda_\pm (\xi) = \xi \pm i $.
We can expand $\Pi_+[\wh A_- ^{-1}  \wh{\wt{f}_1}] (\xi)$ such that
\begin{equation}\label{eqn:hat_psi_expan}
     \wh{\psi_+ }(\xi) = \sum_{k=1}^n \wh A_+ ^{-1}(\xi) \wh {\Lambda}_+^{-k}(\xi)\Pi'[\wh {\Lambda}_+^{k-1}\wh A_- ^{-1}  \wh{\wt{f}_1}]  + \wh A_+ ^{-1}(\xi)\wh {\Lambda}_+^{-n}(\xi)\Pi_+[\wh {\Lambda}_+^{n}\wh A_- ^{-1}  \wh{\wt{f}_1}] (\xi),
\end{equation}
where $\Pi'$ is defined by
\begin{equation}
    \Pi'[\wh h] =  \frac{i}{2\pi}\int_\R  \wh h(\eta) \,d\eta.
\end{equation}
\rev{The remainder term satisfies}
\begin{equation}\label{eqn:psi-hat-remainder-reg}
    \wh A_+ ^{-1}(\xi)\wh {\Lambda}_+^{-n}(\xi)\Pi_+[\wh {\Lambda}_+^{n}\wh A_- ^{-1}  \wh{\wt{f}_1}] (\xi) \in \mathcal{F}(H^{n-\frac{1}{2}+\varepsilon}_{\overline{\R_+}}(\R)),
\end{equation}
\rev{which has optimal regularity. The remaining terms}
\begin{equation}\label{eqn:psi-hat-singular-factors}
    \wh A_+ ^{-1}(\xi) \wh {\Lambda}_+^{-k}(\xi) = (\xi + i)^{-k+\frac{1}{2}}\in \mathcal{F}(H^{k-\frac{3}{2} + \varepsilon}_{\overline{\R_+}}(\R)), \quad k=1,2,\cdots, n,
\end{equation}
\rev{are the more singular terms and can be computed explicitly:}
\begin{equation}\label{eqn:psi-singular-basis-ft}
    \mathcal{F}^{-1}( \wh A_+ ^{-1}(\xi) \wh {\Lambda}_+^{-k}(\xi)) (x)  =  \frac{1}{2\pi}\int_\R  (\xi+i)^{-k+\frac{1}{2}} e^{-i x \xi}\,d\xi = \frac{e^{(1/4-k/2)i\pi}}{\Gamma(k-\frac{1}{2})}\chi_{[0,\infty)}(x)e^{-x} x^{k-\frac{3}{2}}.
\end{equation}
\rev{Using~\eqref{eqn:hat_psi_expan} and~\eqref{eqn:psi-singular-basis-ft}, the solution can be expressed as}
\begin{equation}\label{eqn:hl-psi-ep}
    \psi_+ (x) = \sum_{k=1}^n c_k \psi_k(x) + \psi_r(x), 
\end{equation}
where
\begin{equation}\label{eqn:psi-r-halfline}
    \psi_r = \mathcal{F}^{-1}(\wh A_+ ^{-1}(\xi)\wh {\Lambda}_+^{-n}(\xi)\Pi_+[\wh {\Lambda}_+^{n}\wh A_- ^{-1}  \wh{\wt{f}_1}]) \in H^{n-\frac{1}{2}+\varepsilon}_{\overline{\R_+}}(\R),
\end{equation}
\rev{is the most regular term by~\eqref{eqn:psi-hat-remainder-reg}}, and $\psi_k$ are given by
\begin{equation}
     \psi_k(x) = \chi_{[0,\infty)}(x)e^{-x} x^{k-\frac{3}{2}},\quad k = 1, 2, \cdots, n.
\end{equation}
The coefficients satisfy the bound
\begin{equation}
    |c_k| \le  M_{k, a} \norm{f_1}_{H^{n+\frac{1}{2}+\varepsilon-a}(\R)} , \quad 0 \le a < n+\frac{1}{2}+\varepsilon-k.
\end{equation}
\rev{Taking \(n=1\) in~\eqref{eqn:hl-psi-ep} gives~\eqref{eqn:psi-halfline-singular}, and proves the lemma.}
\end{proof}

\begin{lemma}\label{lem:B_regular_op}
Let $f_2 \in H^{s}(\R_+)$ with $s > -\frac12$. Then there exists a unique solution $v \in H^{s+2}(\R_+)$ to the boundary value problem \eqref{eqn:pdo-B}. Furthermore, there exists a constant $C > 0$ such that
    \begin{equation}
        \norm{v}_{H^{s+2}(\R_+)} \le C \paren{ \norm{f_2}_{H^{s}(\R_+)} + |g|}.
    \end{equation}
\end{lemma}
\begin{proof}
The existence of the solution $v$ can be established by representing it as
\begin{equation}
    v(x) = \frac{1}{2\pi} \int_{\mathbb{R}} \frac{\wh{\wt f}_2(\xi) e^{ix\xi}}{1+\xi^2}\,d\xi + c e^{-x},
\end{equation}
where $\wt f_2\in H^s(\R)$ is an extension of $f_2$ to $\R$, and the constant $c$ is uniquely determined by the boundary condition specified in \eqref{eqn:pdo-B}.
\end{proof}

We are also interested in the coupling of \eqref{eqn:pdo-A} and \eqref{eqn:pdo-B}, where the function $f_2$ may include the solution $\psi$ and has a singular behavior. This will also affect the regularity of the solution $v$.
\begin{lemma}\label{lem:U_singular}
Let $f_2 = f_s + f_r$, where $f_s(x) = ce^{-x}x^{-\frac12} \in H^{-\frac12+\varepsilon}(\R_+)$ with $c\in \R$, and $f_r \in H^{\frac12+\varepsilon}(\R_+)$ for $|\varepsilon| < \frac12$. Then the solution $v$ to $Bv=f_2$ admits the decomposition
\begin{equation}
     v (x) = -\frac{4}{3}c e^{-x} x^{\frac32} + v_r(x), 
\end{equation}
where the remainder $v_r \in H^{\frac52+\varepsilon}(\R_+)$ satisfies the bound
\begin{equation}
    \norm{v_r}_{H^{\frac52+\varepsilon}(\R_+)} \le C \paren{ \norm{f_r}_{H^{\frac12+\varepsilon}(\R_+)} + |c| + |g|}.
\end{equation}
\end{lemma}
\begin{proof}
Let $v_s(x) = -\frac{4}{3} e^{-x} x^{\frac32}$ be the singular part. A direct computation shows that the operator $B = \mathcal{I} - \partial_{xx}$ acting on $v_s$ yields
\begin{equation}
    B v_s = v_s - v_s'' = e^{-x}x^{-\frac12} - 4 e^{-x}x^{\frac12}.
\end{equation}
Substituting $v = c v_s + v_r$ into the governing equation $Bv = f_s + f_r$, we find that the remainder $v_r$ satisfies
\begin{equation}
    B v_r = f_s + f_r - c B v_s = f_r + 4c e^{-x} x^{\frac12},
\end{equation}
with the boundary condition $\partial_x v_r(0) = g$. Since $f_r$ and $4c e^{-x} x^{\frac12}$ both belong to $H^{\frac12+\varepsilon}(\R_+)$ for $\abs{\varepsilon} < \frac12$, the result follows by applying the elliptic regularity from Lemma~\ref{lem:B_regular_op} to $v_r$.
\end{proof}

\subsection{Open arc case}
We now investigate higher regularity properties of the solution. 
In particular, when the given data are more regular, we aim to derive the asymptotic behavior of the boundary singularities. 
From this point on, we assume that the open arc $S$ is smooth and parameterized by arc length, that is, 
$\bm{X}\in C^\infty(I)$ and $|\bm{X}'(s)|=1$. 
Although a classical solution to~\eqref{eqn:arc-pde} only requires the curve to be $C^2$ in order for the Laplace-Beltrami operator to be defined pointwise, we do not pursue this level of generality here.
Throughout the discussion, we denote $u \circ \bm{X}$ simply by $u$, with the understanding that all functions are pulled back to the parameter domain $I$.
The single-layer operator on the open arc is then expressed as
\begin{equation}
    (V_S \psi)(s) = -\frac{1}{2\pi} \int_I \log|\bm{X}(s) - \bm{X}(\eta)|\, \psi(\eta)\, d\eta.
\end{equation}
We also introduce the corresponding operator on the straight interval,
\begin{equation}\label{eqn:V-psi}
    (V \psi)(s) = -\frac{1}{2\pi} \int_{\mathbb{R}} \log|s-\eta|\, \psi(\eta)\, d\eta.
\end{equation}
For convenience, let $R = V - V_S$ denote the difference between the two operators.
Under this parametrization, the system~\eqref{eqn:arc-bie} can be equivalently written as
\begin{align}
    -\partial_{ss} U + \psi &= f, \quad\text{in } I, \label{eq:U-psi-eq1}\\
    V_S \psi - U &= h, \quad\text{in } I, \label{eq:U-psi-eq2}\\
    \partial_s U(\pm 1) &= \pm g(\pm 1). \label{eq:U-psi-bc}
\end{align}

We take a partition of unity $\{\varphi_n\}_{n=1}^3$ such that $\varphi_1(s)$ and $\varphi_3(s)$ are equal to $1$ in neighborhoods of the endpoints $-1$ and $1$ respectively, and $\varphi_2(s)$ is supported away from the boundary $\partial I$, satisfying
\begin{equation}
    0 \le \varphi_n \le 1, \quad
    \sum_{n=1}^3 \varphi_n(s) = 1, \quad \forall\, s \in I.
\end{equation}
We express the solutions as
\begin{equation}
    \psi(s) = \sum_{n=1}^3 \vph_n(s)\psi(s), \quad U(s) = \sum_{n=1}^3 \vph_n(s)U(s), \quad s\in I.
\end{equation}
We will analyze the local solutions $\vph_n\psi, \vph_n U$ and then patch them up to obtain the regularity of the solution. 
\begin{theorem}\label{thm:global}
    Given $f \in H^{-\frac{1}{2}+\varepsilon}(I)$, $h\in H^{\frac{1}{2}+\varepsilon}(I)$ and $g$ with $\norm{g}_{\partial I} <+\infty$ for $|\varepsilon| < \frac{1}{2}$. Then the solutions satisfy $\psi \in H^{-\frac{1}{2} + \varepsilon}_{\overline{I}}(\R)$ and $U \in H^{\frac{3}{2} + \varepsilon}(I)$ and 
    \begin{equation}
        \norm{U}_{H^{\frac{3}{2}+\varepsilon}(I)}+\norm{\psi}_{H^{-\frac{1}{2}+\varepsilon}(\R)} \le C \paren{ \norm{f}_{H^{-\frac{1}{2}+\varepsilon}(I)} +\norm{h}_{H^{\frac{1}{2}+\varepsilon}(I)}+ \norm{g}_{\partial I}}.
    \end{equation}
\end{theorem}
\begin{proof}
Since $f\in H^{-\frac{1}{2}+\varepsilon}(I)$ belongs to $H^{-1}_{\overline{I}}(\R)$ and $h\in H^{\frac{1}{2}+\varepsilon}(I)$ belongs to $H^\frac{1}{2}(I)$, Theorem~\ref{thm:solvability} applies and we have $U\in H^1(I)$ and $\psi\in H^{-\frac{1}{2}}_{\overline{I}}(\R)$ with
    \begin{equation}
    \begin{aligned}
        \paren{\norm{U}_{H^1(I)} + \norm{\psi}_{H^{-\frac{1}{2}}(\R)} } 
        \le C \paren{ \norm{f}_{H^{-\frac{1}{2}+\varepsilon}(\R)}  + \norm{h}_{H^{\frac{1}{2}+\varepsilon}(I)} + \norm{g}_{\partial I} }.
    \end{aligned}    
    \end{equation}

    We first consider the interior regularity of $\psi$. 
    We choose a smooth cutoff function $\vph \in C_c^\infty(I)$. We consider the equation
    \begin{equation}\label{eqn:Avphpsi}
        A(\vph \psi) = (A - V)(\vph\psi)+ [V, \vph] \psi + \vph R\psi + \vph V_S \psi,
    \end{equation}
    where the commutator $[V,\vph]\psi$ is defined as    
    \begin{equation}
        [V,\vph]\psi = V(\vph \psi) - \vph V\psi.
    \end{equation}
    The operator defined by pointwise multiplication by $\vph$ is a pseudo-differential operator of order $0$, and it holds that
    \begin{equation}
        \norm{\vph u}_{H^s(\R)} \le C \norm{u}_{H^s(\R)}, \quad \forall s\in\R.
    \end{equation}
    The complete symbol of the operator $V$ is given by $\wh V(\xi) = \frac{1}{2|\xi|}$. Let $V_0$ be the pseudo-differential operator with symbol $\frac{1-\rho(\xi)}{2|\xi|}$ where $\rho(\xi) \in C_c^\infty(\R)$ is a cutoff function being $1$ for $|\xi| < \frac{1}{2}$ and $\rho(\xi)=0$ for $|\xi|>1$. 
    Then $V-V_0$ defines a smoothing operator and $V-V_0:H^s(\R)\to H^r(\R)$ is bounded for all $s < r$.
    We also have
    \begin{equation}\label{eqn:AminusV0-symbol}
        |\wh A(\xi) - \wh V_0(\xi)| \le C (1 + |\xi| )^{-3},
    \end{equation}
    \rev{By~\eqref{eqn:AminusV0-symbol}, \(A-V_0\) is a pseudo-differential operator of order \(-3\), while \(V-V_0\) is smoothing. Hence,}
    \begin{equation}\label{eqn:AminusV-est}
    \begin{aligned}
        \norm{(A-V)\vph\psi}_{H^s(\R)} &\le \norm{(A-V_0)\vph\psi}_{H^s(\R)} + \norm{(V-V_0)\vph\psi}_{H^s(\R)} \\
        &\le C\norm{\vph\psi}_{H^{s-3}(\R)} \le C \norm{\psi}_{H^{s-3}(\R)}, 
    \end{aligned}
    \end{equation}
    \rev{In particular, this term is bounded by \(C\norm{\psi}_{H^{-\frac{1}{2}}(\R)}\) whenever \(s\le \frac{5}{2}\).}
    \rev{For the commutator term, we split}
    \begin{equation}
        [V,\vph]\psi = [V_0, \vph]\psi + [V-V_0, \vph]\psi.
    \end{equation}
    \rev{Since \(V_0\) has order \(-1\) and multiplication by \(\vph\) has order \(0\), the commutator \([V_0,\vph]\) has order \(-2\); moreover, \([V-V_0,\vph]\) is smoothing. Therefore}
    \begin{equation}\label{eqn:V-comm-est}
    \begin{aligned}
        \norm{[V,\vph]\psi}_{H^s(\R)} &\le \norm{[V_0,\vph]\psi}_{H^s(\R)} + \norm{[V-V_0, \vph]\psi}_{H^s(\R)}  \\
        &\le C \norm{\psi}_{H^{s-2}(\R) }, 
    \end{aligned}
    \end{equation}
    \rev{In particular, this term is bounded by \(C\norm{\psi}_{H^{-\frac{1}{2}}(\R)}\) whenever \(s\le \frac{3}{2}\).}
    Now we consider $ \vph R\psi$, which can be written as
    \begin{equation}
        \begin{aligned}
            \vph R\psi (s) = -\frac{1}{2\pi} \vph(s)\int_{I} \log\abs{\frac{\Delta\bm{X}}{s-t} }\psi(t)\,dt,
        \end{aligned}
    \end{equation}
    where $\Delta\bm{X} = \bm{X}(s)-\bm{X}(t)$.
    Note that $\forall N$, $\bm{X}(t) - \bm{X}(s) = \sum_{j=1}^N \frac{\bm{X}^{(j)}(s)}{j!} (t-s)^j + \mathcal{O}(|t-s|^{N+1})$. 
    So $\frac{\bm{X}(t)-\bm{X}(s)}{t-s} \in C^\infty(\overline{I} \times \overline{I})$ and is nowhere vanishing, which implies the kernel $\log\abs{\frac{\Delta\bm{X}}{s-t}}$ is smooth.
    \rev{Since $\psi\in H^{-\frac{1}{2}}_{\overline{I}}(\R)\subset H_{\overline{I}}^{-1}(\R)$, we may write}
    \begin{equation}
        \psi(s) = \psi_0 + \partial_s F (s), \quad s\in I, \quad \psi_0 = \int_I \psi(s)\,ds, \quad F\in L^2 (I),
    \end{equation}
    \rev{where \(F\) is chosen to have zero mean. The one-dimensional primitive estimate, together with \(H^{-\frac12}_{\overline I}(\mathbb R)\hookrightarrow H^{-1}(I)\), yields}
    \begin{equation}
        \rev{|\psi_0|+\norm{F}_{L^2(I)}
        \le C\norm{\psi}_{H^{-1}(I)}
        \le C\norm{\psi}_{H^{-\frac12}(\mathbb R)}.}
    \end{equation}
    \rev{Integrating by parts in this decomposition gives}
    \begin{equation}
    \begin{aligned}
        \int_{I} \log\abs{\frac{\Delta \bm{X}}{s-t} }\psi(t)\,dt = \psi_0\int_{I} \log\abs{\frac{\Delta \bm{X}}{s-t} } \,dt - \int_I \partial_t   \log\abs{\frac{\Delta \bm{X}}{s-t} } F(t)\,dt.
    \end{aligned}
    \end{equation}
    and for every $k\in \N$, by Cauchy-Schwarz inequality,
    \begin{equation}
    \begin{aligned}
        &\norm{\partial_s^k\int_I  \partial_t  \log\abs{\frac{\Delta \bm{X}}{s-t} } F(t)\,dt}_{L^2(I)}^2 = \int_I \paren{\int_I \partial_s^k \partial_t   \log\abs{\frac{\Delta \bm{X}}{s-t} } F(t)\,dt}^2\,ds\\
        &\le \int_I \int_I \paren{ \partial_s^k \partial_t   \log\abs{\frac{\Delta \bm{X}}{s-t} }}^2\,dt   \,ds \norm{F}_{L^2(I)}^2 \le C \norm{F}_{L^2(I)}^2 \le C \norm{\psi}_{H^{-\frac{1}{2}}(\R)}\rev{^2.}
    \end{aligned}
    \end{equation}
    This implies that $R\psi\in H^{k}(I)$ for every $k \in \N$.
    Thus, let $\wt{R\psi}$ be an extension of $R\psi$ with $\norm{\wt{R\psi}}_{H^k(\R)}\le 2\norm{R\psi}_{H^k(I)}$ we have
    \begin{equation}
    \begin{aligned}
        \norm{\vph R\psi}_{H^k(\R)} & = \norm{\vph \wt{R\psi}}_{H^k(\R)} \le C\norm{\wt{R\psi}}_{H^k(\R)} \\
        &\le C\norm{R\psi}_{H^k(I)} \le   C \norm{\psi}_{H^{-\frac{1}{2}}(\R)}, \quad \forall k\in\N.
    \end{aligned}
    \end{equation}
    By \rev{Sobolev }interpolation\rev{ \cite[Appendix~B]{McLean2000}}, we have
    \begin{equation}
    \begin{aligned}
        \norm{\vph R\psi}_{H^s(\R)} & \le  C \norm{\psi}_{H^{-\frac{1}{2}}(\R)}, \quad \forall s \ge 0.
    \end{aligned}
    \end{equation}
    Recall that $V_S\psi = U+h$ on $I \supseteq \text{supp }\vph$. 
    Let $\wt{U}$, $\wt{h}$ be extensions of $U$ and $h$ with $\norm{\wt{U}}_{H^1(\R)}\le 2\norm{U}_{H^1(I)}$ and $\norm{\wt{h}}_{H^{\frac{1}{2}+\varepsilon}(\R)} \le 2\norm{h}_{H^{\frac{1}{2}+\varepsilon}(I)}$.
    Then, $\vph V_S\psi = \vph (U+h) = \vph (\wt{U}+\wt{h})$ and
    \begin{equation}
    \begin{aligned}
        \norm{\vph V_S\psi}_{H^{\frac{1}{2}+\varepsilon}(\R)} &= \norm{\vph (\wt{U} + \wt{h})}_{H^{\frac{1}{2}+\varepsilon}(\R)} \le C\norm{\wt{U} + \wt{h}}_{H^{\frac{1}{2}+\varepsilon}(\R)}\\
        &\le C \paren{\norm{U}_{H^1(I)} + \norm{h}_{H^{\frac{1}{2}+\varepsilon}(I)} },
    \end{aligned}
    \end{equation}
    In summary, for every $|\varepsilon|<\frac{1}{2}$,
    \begin{equation}
    \begin{aligned}
        &\norm{ (A - V)(\vph\psi)+ [V, \vph] \psi + \vph R\psi + \vph V_S \psi }_{H^{\frac{1}{2}+\varepsilon}(\R)} \\
        &\le C \paren{\norm{U}_{H^1(I)} +\norm{h}_{H^{\frac{1}{2}+\varepsilon}(I)}+ \norm{\psi}_{H^{-\frac{1}{2}(\R)}}}.
    \end{aligned}
    \end{equation}
    The inverse operator $A^{-1}$ has symbol $(1 + |\xi|^2)^{\frac{1}{2}}$ and is of order $1$. 
    Using~\eqref{eqn:Avphpsi}, we have
    \begin{equation}
        \norm{\vph\psi}_{H^{-\frac{1}{2}+\varepsilon}(\R)} \le C \paren{\norm{U}_{H^1(I)}  + \norm{h}_{H^{\frac{1}{2}+\varepsilon}(I)} + \norm{\psi}_{H^{-\frac{1}{2}(\R)}}}.
    \end{equation}
    
    To study the regularity of $\psi$ near the boundary $\partial I$, we shift the left endpoint $x=-1$ to the origin under the transformation $x'=1+x$ and let $\vph$ be a smooth cutoff function localized at the origin, satisfying $\text{supp }\vph\supset (-\delta, \delta)$ for some $\delta>0$ and $\text{supp }\vph \cap\overline{\R_+} \subset [0,1)$.
    We now consider the following equation on the half-line,
    \begin{equation}\label{eqn:loc_psi_bdr}
        p_+ A(\vph \psi) = p_+\paren{ (A - V)(\vph\psi)+ [V, \vph] \psi + \vph R\psi + \vph V_S \psi }.
    \end{equation}
    \rev{Using the \(A-V\) estimate in~\eqref{eqn:AminusV-est} and the commutator estimate in~\eqref{eqn:V-comm-est} we obtain}
    \begin{equation}\label{eqn:p_A-V_Vphi}
    \begin{aligned}
        \norm{p_+\paren{ (A-V)(\vph\psi) + [V,\vph]\psi} }_{H^{s}(\R_+)} &\le \norm{ (A-V)(\vph\psi) + [V,\vph]\psi}_{H^{s}(\R)} \\
        &\le C\norm{\psi}_{H^{s-2}(\R)}  ,        
    \end{aligned}
    \end{equation}
    for every $s\ge \frac{3}{2}$,     and
    \begin{equation}\label{eqn:p_phi_Rpsi_integer}
    \begin{aligned}
        & \norm{p_+ \vph R\psi}_{H^k(\R_+)} = \norm{p_+ \vph \wt{R\psi}}_{H^k(\R_+)} \le C\norm{\vph \wt{R\psi}}_{H^k(\R)} \\
        &\le C\norm{\wt{R\psi}}_{H^k(\R)} \le C\norm{R\psi}_{H^k(I)}  \le  C \norm{\psi}_{H^{-\frac{1}{2}}(\R)}, \quad \forall k \in\N.
    \end{aligned}
    \end{equation}
    By \rev{Sobolev }interpolation\rev{ \cite[Appendix~B]{McLean2000}}, we have
    \begin{equation}\label{eqn:p_phi_Vs}
        \norm{p_+ \vph R\psi}_{H^s(\R_+)} \le  C \norm{\psi}_{H^{-\frac{1}{2}}(\R)}, \quad \forall s\ge 0.
    \end{equation}
    Since $\text{supp }\vph \cap \R_+ \subset I$, it holds that $p_+\vph V_S\psi = p_+\vph( U + h) = p_+ \vph (\wt{U}+\wt{h})$.
    Then
    \begin{equation}
        \norm{p_+\vph V_S\psi}_{H^{\frac{1}{2}+\varepsilon}(\R_+)} \le \norm{\vph( \wt{U} + \wt{h})}_{H^{\frac{1}{2}+\varepsilon}(\R)}  \le C\paren{ \norm{U}_{H^1(I)}  + \norm{h}_{H^{\frac{1}{2}+\varepsilon}(I)} }.
    \end{equation}
    Thus, by Lemma~\ref{lem:A_global}, we obtain
    \begin{equation}
        \norm{\vph\psi}_{H^{-\frac{1}{2}+\varepsilon}(\R)} \le C \paren{\norm{U}_{H^1(I)} + \norm{h}_{H^{\frac{1}{2}+\varepsilon}(I)}  + \norm{\psi}_{H^{-\frac{1}{2}}(\R)}}.
    \end{equation}
    Patching up local solutions and using Theorem~\ref{thm:solvability}, we obtain
    \begin{equation}\label{eqn:psi-I-reg}
    \begin{aligned}
        \norm{\psi}_{H^{-\frac{1}{2}+\varepsilon}(\R)} & \le C \paren{\norm{U}_{H^1(I)}  + \norm{h}_{H^{\frac{1}{2}+\varepsilon}(I)} + \norm{\psi}_{H^{-\frac{1}{2}}(\R)} } \\
        &\le C \paren{ \norm{f}_{H^{-\frac{1}{2}+\varepsilon}(\R)} + \norm{g}_{\partial I}  + \norm{h}_{H^{\frac{1}{2}+\varepsilon}(I)} }.
    \end{aligned}    
    \end{equation}

    To study the regularity of $U$, we use the cutoff function $\vph$ with $\text{supp }\vph \subset I$ and consider the equation
    \begin{equation}
        B(\vph U) = \vph U - \partial_{ss} \vph U - 2\partial_s \vph\partial_s U - \vph \partial_{ss} U.
    \end{equation}
    Recalling that $\partial_{ss} U = \psi - f$ on $I$ and $\text{supp } \partial_s^k\vph \subset I$ for $k \ge 0$, we have
    \begin{equation}
    \begin{aligned}
        &\norm{\vph U - \partial_{ss} \vph U - 2\partial_s \vph\partial_s U - \vph \partial_{ss} U}_{H^{-\frac{1}{2}+\varepsilon}(\R)} \\
        &= \norm{\vph (\wt{U}-\psi + f) - \partial_{ss} \vph \wt{U} - 2\partial_s \vph\partial_s \wt{U}}_{H^{-\frac{1}{2}+\varepsilon}(\R)} \\
        &\le C\paren{ \norm{\wt{U}}_{H^{\frac{1}{2}+\varepsilon}(\R)} +  \norm{\psi}_{H^{-\frac{1}{2}+\varepsilon}(I)} + \norm{f}_{H^{-\frac{1}{2}+\varepsilon}(I)} }\\
        &\le C\paren{ \norm{U}_{H^{1}(I)} +  \norm{\psi}_{H^{-\frac{1}{2}+\varepsilon}(\R)} + \norm{f}_{H^{-\frac{1}{2}+\varepsilon}(I)} }.
    \end{aligned}
    \end{equation}
    where we have used $\norm{\vph f}_{H^{-\frac{1}{2}+\varepsilon}(\R)}\le C \norm{\vph f}_{H^{-\frac{1}{2}+\varepsilon}(I)}$ since $\text{supp }(\vph f) \subset I$.
    The inverse operator $B^{-1}$ has the symbol $(1 + |\xi|^2)^{-1}$ and is of order $-2$. Hence,
    \begin{equation}
        \norm{\vph U}_{H^{\frac{3}{2}+\varepsilon}(I)} \le \norm{\vph U}_{H^{\frac{3}{2}+\varepsilon}(\R)} \le C\paren{ \norm{U}_{H^{1}(I)} +  \norm{\psi}_{H^{-\frac{1}{2}+\varepsilon}(\R)} + \norm{f}_{H^{-\frac{1}{2}+\varepsilon}(I)} }.
    \end{equation}
    To study the boundary regularity of $U$, we use the same cutoff function $\vph$ and consider the following equation on the half-line:
    \begin{equation}
        p_+B(\vph U) = p_+ \paren{ \vph U - \partial_{ss} \vph U - 2\partial_s \vph\partial_s U - \vph \partial_{ss} U},
    \end{equation}
    with boundary condition $\partial_s(\vph U)(0) = \partial_s U(0) = g(0)$.
    For the right-hand side, we have
    \begin{equation}
    \begin{aligned}
        &\norm{p_+ \paren{ \vph U - \partial_{ss} \vph U - 2\partial_s \vph\partial_s U - \vph \partial_{ss} U}}_{H^{-\frac{1}{2}+\varepsilon}(\R_+)} \\
        &=\norm{p_+ \paren{ \vph (\wt{U}-\psi+f) - \partial_{ss} \vph \wt{U} - 2\partial_s \vph\partial_s \wt{U}}}_{H^{-\frac{1}{2}+\varepsilon}(\R_+)} \\
        &\le \norm{\vph (\wt{U}-\psi+f) - \partial_{ss} \vph \wt{U} - 2\partial_s \vph\partial_s \wt{U} }_{H^{-\frac{1}{2}+\varepsilon}(\R)} \\
        &\le C\paren{ \norm{\wt{U}}_{H^{\frac{1}{2}+\varepsilon}(\R)}  +\norm{\psi}_{H^{-\frac{1}{2}+\varepsilon}(\R)} + \norm{f}_{H^{-\frac{1}{2}+\varepsilon}(I)} }\\
        &\le C\paren{ \norm{U}_{H^{1}(I)} +  \norm{\psi}_{H^{-\frac{1}{2}+\varepsilon}(\R)} + \norm{f}_{H^{-\frac{1}{2}+\varepsilon}(I)} }.
    \end{aligned}
    \end{equation}
    By Lemma~\ref{lem:B_regular_op} , we have
    \begin{equation}
        \norm{\vph U}_{H^{\frac{3}{2}+\varepsilon}(I)} \le C\paren{ \norm{U}_{H^{1}(I)} +  \norm{\psi}_{H^{-\frac{1}{2}+\varepsilon}(\R)} + \norm{f}_{H^{-\frac{1}{2}+\varepsilon}(I)} + \norm{g}_{\partial I} }.
    \end{equation}
    Patching up local solutions and using Theorem~\ref{thm:solvability}, we obtain
    \begin{equation}
    \norm{U}_{H^{\frac{3}{2}+\varepsilon}(I)} \le C \paren{ \norm{f}_{H^{-\frac{1}{2}+\varepsilon}(I)} + \norm{g}_{\partial I}  + \norm{h}_{H^{\frac{1}{2}+\varepsilon}(I)} }.
    \end{equation}
This completes the proof.
\end{proof}

\begin{theorem}
Given $f\in H^{\frac{1}{2}+\varepsilon}(I)$, $h\in H^{\frac{3}{2}+\varepsilon}(I)$, and $g$ with $\norm{g}_{\partial I} < \infty$, the solutions $\psi$ and $U$ have the form
\begin{equation}\label{eqn:U_final}
    U(x) = \sum_{p\in\partial I} c_2(p) \vph(|x-p|) |x-p|^{\frac{3}{2}} + U_r(x),\quad U_r \in H^{\frac{5}{2}+\varepsilon}(I),
\end{equation}
and
\begin{equation}\label{eqn:psi+_final}
    \psi(x) = \chi_{I}(x)\sum_{p\in\partial I} c_1(p) \vph(|x-p|) |x-p|^{-\frac{1}{2}}  + \psi_r(x),\quad \psi_r \in H^{\frac{1}{2}+\varepsilon}_{\overline{I}}(\R),
\end{equation}
where $c_1(p),c_2(p)$ are constants depending on $p$ and $\vph(d)$ is a smooth cut-off function.
\end{theorem}
\begin{proof}
    We continue to consider equation~\eqref{eqn:loc_psi_bdr} for the local solution $\vph \psi$.
    Applying Theorem~\ref{thm:global}, we have
    \begin{equation}
    \begin{aligned}
        \norm{p_+\vph V_S\psi}_{H^{\frac{3}{2}+\varepsilon}(\R_+)} &\le \norm{\vph (\wt{U}+\wt{h})}_{H^{\frac{3}{2}+\varepsilon}(\R)} \le C\norm{\wt{U} + \wt{h}}_{H^{\frac{3}{2}+\varepsilon}(\R)}\\
        &\le C\paren{\norm{U}_{H^{\frac{3}{2}+\varepsilon}(I)} + \norm{h}_{H^{\frac{3}{2}+\varepsilon}(I)}}.
    \end{aligned}
    \end{equation}
    \rev{We first control the \(R\)-term in~\eqref{eqn:loc_psi_bdr}. Since \(R=V-V_S\), this contribution is \(p_+\vph R\psi\). Choosing an integer \(k>\frac32+\varepsilon\) in~\eqref{eqn:p_phi_Rpsi_integer} and using the embedding \(H^k(\R_+)\hookrightarrow H^{\frac32+\varepsilon}(\R_+)\), we obtain}
    \begin{equation}
    \begin{aligned}
        \rev{\norm{p_+\vph R\psi}_{H^{\frac32+\varepsilon}(\R_+)}
        \le C\norm{\psi}_{H^{-\frac12}(\R)}
        \le C\norm{\psi}_{H^{-\frac12+\varepsilon}(\R)}.}
    \end{aligned}
    \end{equation}
    \rev{Next, we take \(s=\frac32+\varepsilon\) in~\eqref{eqn:p_A-V_Vphi} and combine the resulting estimate with the bound above for \(p_+\vph R\psi\) and the preceding bound for \(p_+\vph V_S\psi\). The triangle inequality then gives}
    \begin{equation}
    \begin{aligned}        
        &\norm{p_+\paren{ (A - V)(\vph\psi)+ [V, \vph] \psi + \rev{\vph R\psi} + \vph V_S \psi } }_{H^{\frac{3}{2}+\varepsilon}(\R_+)} \\
        & \le C \paren{ \norm{U}_{H^{\frac{3}{2}+\varepsilon}(I)} + \norm{\psi}_{H^{-\frac{1}{2}+\varepsilon}(\R) } + \norm{h}_{H^{\frac{3}{2}+\varepsilon}(I)} }.
    \end{aligned}
    \end{equation}
    By Lemma~\ref{lem:psi_singular}, if $\text{supp }\vph \cap \{0\} \neq \emptyset$, the local solution has the form
    \begin{equation}
        \vph\psi (x) = c \chi_{[0,\infty)} (x) e^{-x} x^{-\frac{1}{2}} + \psi_r(x), \quad \psi_r \in H^{\frac{1}{2}+\varepsilon}_{\overline{\R_+}}(\R),
    \end{equation}
    and if $\text{supp }\vph \subset I$, the local solution $\vph\psi\in H^{\frac{1}{2}+\varepsilon}_{\overline{I}}(\R)$. Summing all local solutions yields \eqref{eqn:psi+_final}.
    
    Using \eqref{eqn:psi+_final} in the local equation for $\vph U$ with $\text{supp }\vph \cap \{0\} \neq \emptyset$ and applying Lemma~\ref{lem:U_singular}, the local solution $\vph U$ has the form
    \begin{equation}
        \vph U(x) = d e^{-x}x^{\frac{3}{2}} + U_r(x), \quad U_r\in H^{\frac{5}{2}+\varepsilon}(\R_+),
    \end{equation}
    where $d$ is some constant.
    For the interior solution, since $\vph \psi \in H^{\frac{1}{2}+\varepsilon}_{\overline{I}}(\R)$ when $\text{supp }\vph\subset I$, we have $\vph U\in H^{\frac{5}{2}+\varepsilon}(I)$.
    Patching together the local solutions for $U$ yields \eqref{eqn:U_final}.
\end{proof}

Mapping the functions back to $S$, by Lemma~\ref{lem:norm-equv-2}, we consequently have
\begin{theorem}\label{thm:reg-u-psi}
Given $f \in H^{\frac{1}{2}+\varepsilon}(S)$, $h\in H^{\frac{3}{2}+\varepsilon}(S)$ and $g$ with $\norm{g}_{\partial S} < +\infty$, the equation \eqref{eqn:arc-bie} has a unique solution $\psi \in H^{-\frac{1}{2}+\varepsilon}_{\overline{S}}(\Gamma)$ with $ \at{\psi}{K}\in H^{\frac{1}{2}+\varepsilon}(K)$ and $U \in H^{\frac{3}{2}+\varepsilon} (S)\cap H_{loc}^{\frac{5}{2}+\varepsilon}(S)$ for every $|\varepsilon| < \frac{1}{2}$ and compact subset $K \Subset S$. In particular, the solutions have the form
\begin{equation}\label{eqn:decomp2}
     U = \sum_{p\in\partial S} c_2(p) \vph(d) d^{\frac{3}{2}} + U_r,\quad \psi = \sum_{p\in\partial S} c_1(p) \vph(d) d^{-\frac{1}{2}}  + \psi_r,\quad \text{on } S,
\end{equation}
with $U_r \in H^{\frac{5}{2}+\varepsilon}(S)$ and $\psi_r \in H^{\frac{1}{2}+\varepsilon}_{\overline{S}}(\Gamma)$. Here, $d$ is the arclength along $\Gamma$ near the boundary $\partial S$, and $\vph$ is a smooth cutoff function.
\end{theorem}

We can further show that $(U,\psi)$ and $(V_\Omega\psi, U)$ are classical solutions.
\begin{proposition}\label{bie-pointwise}
If $f \in H^{\frac{1}{2}+\varepsilon}(S)$,  $h\in H^{\frac{3}{2}+\varepsilon}(S)$ and $g$ with $\norm{g}_{\partial S} < +\infty$ for $\varepsilon\in(0,\frac{1}{2})$, then $U\in C^2(S)\cap C^1(\overline{S})$, $\psi\in C(S)$, and the equation~\eqref{eqn:arc-bie} is satisfied pointwise.
\end{proposition}
\begin{proof}
Since $H^s(\mathbb{R})\hookrightarrow C^{0,\alpha}(\mathbb{R})$ for 
$0<\alpha<s-\frac12$, Theorem~\ref{thm:reg-u-psi} implies that
$U\in C^{1,\alpha}(S)$. Moreover, for every compact subset $K\Subset S$,
$\left.U\right|_K\in C^{2,\alpha}(K)$ and 
$\left.\psi\right|_K\in C^{0,\alpha}(K)$
for any $\alpha\in(0,\varepsilon)$.
Consequently, $U\in C^2(S)\cap C^1(\overline S)$ and $\psi\in C(S)$.
In addition, by Lemma~\ref{lem:vs-reg}, we have
$V_S\psi\in H^{\frac12+\varepsilon}(S)\hookrightarrow C^{0,\alpha}(S)$.
Since the pair $(U,\psi)$ satisfies the weak formulation~\eqref{eqn:weak-eqn}, it follows that \eqref{eqn:arc-bie} is satisfied pointwise.
\end{proof}
\begin{proposition}\label{pde-pointwise}
If $f \in H^{\frac{1}{2}+\varepsilon}(S)$ and $g$ with $\norm{g}_{\partial S} < +\infty$ for $\varepsilon\in(0,\frac{1}{2})$, then $u=V_\Omega\psi\in C^2(\R^2\setminus \overline{S})\cap C(\R^2)$ and $U\in C^2(S)\cap C^1(\overline{S})$, and the equation~\eqref{eqn:arc-pde} is satisfied pointwise.
\end{proposition}
\begin{proof}
Since the kernel $\log|x-y|$ is smooth at $x\in\mathbb{R}^2\setminus\overline S$
for $y\in S$, it follows immediately that
$V_\Omega\psi\in C^2(\mathbb{R}^2\setminus\overline S)$, and the Laplace
equation~\eqref{eqn:arc-pde-a} is satisfied pointwise in
$\mathbb{R}^2\setminus\overline S$.
The decay condition~\eqref{eqn:arc-pde-e} follows from the compatibility
condition~\eqref{eqn:compatible}, while the boundary condition
\eqref{eqn:arc-pde-d} was established in the previous proposition.
Therefore, it remains to verify that~\eqref{eqn:arc-pde-b}--%
\eqref{eqn:arc-pde-c} hold pointwise.

By Theorem~\ref{thm:reg-u-psi} and \cite[Theorem~6.12]{McLean2000}, since
$\psi\in H_{\overline S}^{-\frac12+\varepsilon}(\Gamma)$ for some
$\varepsilon\in(0,\frac12)$, it follows that for any bounded open set
$O\subset\mathbb{R}^2$ and any cutoff function
$\chi\in C_c^\infty(\mathbb{R}^2)$ with $\mathrm{supp}\,\chi\subset O$, we have $
\chi V_\Omega\psi \in H^{1+\varepsilon}(O)\hookrightarrow C^{0,\alpha}(O)$ where $\alpha<\varepsilon.$
This implies that $V_\Omega\psi\in C(\mathbb{R}^2)$ and $\at{u}{S} = V_S\psi = U$ pointwise on $S$.

We now prove the jump relation
$\jump{\partial_{\bm n}V_\Omega\psi}=\psi$ pointwise on $S$.
Fix a point $x\in S$ and choose $r>0$ sufficiently small such that
$\mathrm{dist}(x,\partial S)\ge 2r$.
Let $\chi\in C_c^\infty(\mathbb{R}^2)$ satisfy
$\chi(y)\equiv1$ for $|x-y|\le r/2$ and $\chi(y)\equiv0$ for $|x-y|\ge r$.
Then $\mathrm{supp}(\chi\psi)\subset S$, and hence
$\chi\psi\in H_{\overline S}^{\frac12+\varepsilon}(\Gamma)
\hookrightarrow C^{0,\alpha}(\Gamma)$.
By \cite[Theorem~6.19]{kress2014}, we obtain $
\jump{\partial_{\bm n}\bigl(V_\Omega(\chi\psi)\bigr)}(x)=\psi(x).$
On the other hand, since $\log|x-y|(1-\chi(y))$ is smooth at $x$, it follows that
$V_\Omega\bigl((1-\chi)\psi\bigr)$ is smooth at $x$, and therefore $
\jump{\partial_{\bm n}V_\Omega((1-\chi)\psi)}(x)=0.$
Combining these results yields $
\jump{\partial_{\bm n}V_\Omega\psi}(x)=\psi(x),$
and consequently~\eqref{eqn:arc-pde-c} holds pointwise on $S$ due to Proposition~\ref{bie-pointwise}.

\end{proof}

\section{Finite element Approximation}\label{sec:fem}
In this section, we introduce a finite element discretization of the system~\eqref{eqn:weak-eqn} and prove an a priori error estimate.
\subsection{Discretization}
The preceding analysis shows that the boundary density $\psi$ exhibits the singular behavior $\psi \sim d^{-\frac{1}{2}}$ near an open endpoint, so that $\psi \in H^{-\frac{1}{2}+\varepsilon}_{\overline S}$ for some $\varepsilon\in\bigl[0,\frac{1}{2}\bigr).$
Hence, a conventional piecewise‐polynomial finite element space can achieve at best $\|\psi - \psi_h\|_{H^{-\frac{1}{2}}(\Gamma)} = O\bigl(h^\varepsilon\bigr).$
Similarly, the solution $U$ behaves like $
U \sim d^{\frac{3}{2}}$ near the endpoint, and satisfies $
U \in H^{\frac{3}{2}+\varepsilon}(\Omega)$ for some $\varepsilon\in\bigl[0,\frac{1}{2}\bigr),$ yielding at most $\|U - U_h\|_{H^1(\Omega)} = O\bigl(h^{\frac{1}{2}+\varepsilon}\bigr).$
Since these singular components are localized at the boundary while the remainder of the solution remains smooth, a more effective remedy is to enrich the finite element space by incorporating basis functions that capture the $d^{-\frac{1}{2}}$ and $d^{\frac{3}{2}}$ singularities explicitly.

Let $N$ be a positive integer and introduce a (quasi‐uniform) partition of $\overline{I}: -1 = x_0 < x_1 < \cdots < x_N = 1$ with mesh‐size $ h = \max_{1\le j\le N}(x_j - x_{j-1})$. Let $\mc T_h = \cup_{j=1}^N \overline{\sigma_j}$ be the computational mesh with $\sigma_j = (x_{j-1}, x_j)$.
First, we define the space of singular elements
\begin{equation}
    S^{q} =  \text{span}\bigl\{ (1+x)^{q}\vph(1+x), (1-x)^{q}\vph(1-x) \bigr\}, 
\end{equation}
where $\vph$ is a smooth cutoff function satisfying $\vph(x) = 1$ for $x \le 1/8$ and $\vph(x) = 0$ for $x \ge 1/4$. For example, we choose $\vph(x) = \eta(8x - 1)$ with
\begin{equation}
\eta(t)
= \frac{\phi(t)}{\phi(t) + \phi(1 - t)},\quad
\phi(t) =
\begin{cases}
\exp\bigl(-1/t\bigr), & t > 0,\\
0,                    & t \le 0,
\end{cases}.
\end{equation}
Let $V_h$ be the continuous piecewise-polynomial finite element space, defined as
\begin{equation}
    V_h = \{\wt v = v \circ\bm{X}^{-1}, v\in C(\overline{I}): \at{v}{\overline{\sigma_j}} \in \mathbb P_1(\overline{\sigma_j}) , j = 1,2,\cdots, N\},
\end{equation}
where $\mathbb P_1(\overline{\sigma_j})$ is the function space of polynomials on $\overline{\sigma_j}$ of degree no greater than $1$.
Define the finite element spaces for the solutions $U$ and $\psi$ respectively.
\begin{align}
    W^1_h &= \{\wt v\in C(\overline{S}): \wt v = v + w \circ \bm{X}^{-1}, v \in V_h, w\in S^{\frac{3}{2}} \},\\
    W^2_h &= \{\wt v\in H^{-\frac{1}{2}}_{\overline{S}}(\Gamma):\at{\wt v}{\overline{S}} =  v+ w \circ \bm{X}^{-1},v \in V_{h}, w\in S^{-\frac{1}{2}} \}.
\end{align}
Let $\mathbb H_h = W^1_h \times W^2_h$. It is evident that we have the inclusions $W^1_h\subset H^1(I)$, $W^2_h\subset H^{-\frac{1}{2}}_{\overline{I}}(R)$ and $\mathbb H_h\subset\mathbb H$.
The Galerkin finite element method is formulated as finding $\bm u_h = (U_h, \psi_h)\in \mathbb H_h$, such that
\begin{equation}\label{eqn:Galerkin}
   a(\bm u_h, \bm v_h) = b(\bm v_h), \quad \forall \bm v_h = (\phi, \zeta)\in \mathbb H_h.
\end{equation}

\subsection{A priori error estimate}
Let $\pi_{L^2}: L^2(S) \to V_h$ be the $L^2$ projection operator such that for $u \in  L^2(S)$, it holds that
\begin{equation}
    \inprod{u - \pi_{L^2}u}{v_h}_{L^2(S)} = 0,\quad \forall v_h\in V_h.
\end{equation}
We \rev{first} state the property of the $L^2$ projection operator for approximating functions in $H^s(S)$, $s \ge 0$ with a piecewise linear finite element space. \rev{The following result follows from \cite[Proposition~1.134]{ErnGuermond2004} together with the Sobolev interpolation~\cite[Appendix~B]{McLean2000}.}
\begin{lemma}\label{lem:l2-error}    
    Let \rev{$t\in[0,1]$, $s\ge t$,} and $u\in H^s(S)$. There \rev{exists} $C>0$ independent of $h$ such that     \begin{equation}
        \norm{u - \pi_{L^2} u}_{H^t(\Gamma)} \le C h^{\min(1,s-t)} \norm{u}_{H^s(\Gamma)}.
    \end{equation}
\end{lemma}

The following result is for the approximation of $L^2$ functions in negative norms.
\begin{lemma}\label{lem:negative-norm}    
Let $t \in (0,1]$. For every $u \in L^2_{\overline{S}}(\Gamma)$, there exists $C>0$ such that
\begin{equation}
    \inf_{v_h\in V_h }\norm{u - v_h}_{H^{-t}_{\overline{S}}(\Gamma)}\le C h^{t} \norm{u - \pi_{L^2}u}_{L^2_{\overline{S}}(\Gamma)}.
\end{equation}
\end{lemma}
\begin{proof}
\rev{Using the dual characterization of \(H^{-t}_{\overline S}(\Gamma)\), the \(L^2\)-orthogonality of the projection \(\pi_{L^2}\), Cauchy--Schwarz, and} Lemma~\ref{lem:l2-error}, we obtain
\begin{equation}
\begin{aligned}
    &\inf_{v_h\in V_h }\norm{u - v_h}_{H^{-t}_{\overline{S}}(\Gamma)} \le \norm{u - \pi_{L^2}u}_{H^{-t}_{\overline{S}}(\Gamma)} = \sup_{\norm{w}_{H^{t}(S)}=1} \inprod{u - \pi_{L^2}u}{w}_{L^2(S)} \\
    & = \sup_{\norm{w}_{H^{t}(S)}=1} \inprod{u - \pi_{L^2}u}{w - \pi_{L^2}w}_{L^2(S)} \\
    & \le \sup_{\norm{w}_{H^{t}(S)}=1} \norm{u - \pi_{L^2}u}_{L^2_{\overline{S}}(\Gamma)}\norm{w - \pi_{L^2}w}_{L^2(S)} \le C h^{t} \norm{u - \pi_{L^2}u}_{L^2_{\overline{S}}(\Gamma)}. 
\end{aligned}
\end{equation}
\end{proof}

The following lemma states the optimal approximation property of the finite element space $\mathbb H_h$.
\rev{In this subsection, \(U_r\) and \(\psi_r\) always denote the regular remainders in the decomposition~\eqref{eqn:decomp2}.}
\begin{lemma}\label{thm:interp}
Let  $\bm u = (U,\psi)\in\mathbb H$ be the solution that admits the decomposition \eqref{eqn:decomp2}. For every $s\in [1,\frac{5}{2}+\varepsilon]$, $t\in [0, \frac{1}{2}+\varepsilon]$, there exists $C>0$ independent of $h$, such that
\begin{equation}
    \inf_{\bm v_h\in \mathbb H_h}\norm{\bm u - \bm v_h}_{\mathbb H} \le C\paren{ h^{\min(1,s-1)}\norm{U_r}_{H^{s}(S)} + h^{\frac{1}{2}+t} \norm{\psi_r}_{H^{t}_{\overline{S}}(\Gamma)}  }.
    \end{equation}
\end{lemma}
\begin{proof}
By Lemma~\ref{lem:l2-error} and~\ref{lem:negative-norm}, we have
\begin{equation}
\begin{aligned}
    & \inf_{\bm v_h\in \mathbb H_h}\norm{\bm u - \bm v_h}_{\mathbb H} \le \inf_{\phi_h\in W_h^1}\norm{U - \phi_h}_{H^1(S)} + \inf_{\zeta_h\in W_h^2} \norm{\psi - \zeta_h}_{H^{-\frac{1}{2}}_{\overline{S}}(\Gamma)} \\
    &\le \inf_{\phi_h\in V_h}\norm{U_r - \phi_h}_{H^1(S)} + \inf_{\zeta_h\in V_h} \norm{\psi_r - \zeta_h}_{H^{-\frac{1}{2}}_{\overline{S}}(\Gamma)} \\
    &\le \norm{U_r - \pi_{L^2} U_r}_{H^1(S)} + C h^{\frac{1}{2}} \norm{\psi_r - \pi_{L^2}\psi_r}_{L^2_{\overline{S}}(\Gamma)} \\
    &\le C\paren{ h^{\min(1,s-1)}\norm{U_r}_{H^{s}(S)} + h^{\frac{1}{2}+t} \norm{\psi_r}_{H^{t}_{\overline{S}}(\Gamma)}  }.
\end{aligned}
\end{equation}
\end{proof}

\begin{lemma}\label{lem:duality}
    Let $\bm u = (U,\psi)\in\mathbb H$ be the exact solution to \eqref{eqn:weak-eqn}. Let $\bm u_h=(U_h,\psi_h)\in\mathbb H_h$ be the numerical solution. For $\sigma\in (0, \frac{1}{2})$, there holds the error estimate in lower norms
    \begin{equation}
        \norm{U-U_h}_{L^2(S)} + \norm{\psi-\psi_h}_{H^{-\frac{1}{2}-\sigma}_{\overline{S}}(\Gamma)} \le C h^\sigma \paren{\norm{\psi-\psi_h}_{H^{-\frac{1}{2}}_{\overline{S}}(\Gamma)} + \norm{U-U_h}_{H^1(S)}}.
    \end{equation}
\end{lemma}
\begin{proof}
    Let $\mc A'$ be the adjoint of $\mc A$, defined as
    \begin{equation}
        \dual{\mc A'\bm u}{\bm v} = \dual{\bm u}{\mc A\bm v} = a(\bm v, \bm u),\quad \forall \bm u, \bm v\in\mathbb H.
    \end{equation}
    It is clear that $(\mc A')^{-1}$ has the same regularity properties as $(\mc A)^{-1}$. 
    Let the norm \rev{$\norm{\cdot}_{Z}$} be 
    \begin{equation}
        \norm{\bm u}_Z = \norm{u_1}_{L^2_{\overline{S}}(\Gamma)} + \norm{u_2}_{H^{\frac{1}{2}+\sigma}(S)}.
    \end{equation}
    Then $Z =L^2_{\overline{S}}(\Gamma)\times H^{\frac{1}{2}+\sigma}(S) \subset \mathbb H'$.
For every $\bm w\in \mathbb H'$ and an arbitrary $\bm v_h\in\mathbb H_h$, we have
\begin{equation}
\begin{aligned}
    &|\inprod{\bm u - \bm u_h}{\bm w}_{L^2}| = | \dual{\mc A(\bm u - \bm u_h)}{(\mc A')^{-1}\bm w}_{L^2} | =| \dual{\mc A(\bm u - \bm u_h)}{(\mc A')^{-1}\bm w - \bm v_h}_{L^2} | \\
    &\le \norm{\mc A(\bm u - \bm u_h)}_{\mathbb H'} \norm{(\mc A')^{-1}\bm w - \bm v_h}_{\mathbb H}  \le  C \norm{\bm u - \bm u_h}_{\mathbb H} \norm{(\mc A')^{-1}\bm w - \bm v_h}_{\mathbb H} .
\end{aligned}
\end{equation}
\rev{It remains to estimate the best approximation of the adjoint solution in \(\mathbb H_h\). Let \((\phi_1,\phi_2)=(\mc A')^{-1}\bm w\).}
\rev{Using the product structure \(\mathbb H_h=W_h^1\times W_h^2\), the inclusions \(V_h\subset W_h^1\) and \(V_h\subset W_h^2\) obtained by setting the singular enrichment coefficients to zero, the \(L^2\)-projection in each component, Lemma~\ref{lem:l2-error} for \(\phi_1\), Lemma~\ref{lem:negative-norm} for \(\phi_2\), and the adjoint regularity estimate together with the definition of \(\norm{\cdot}_Z\), we obtain}
\begin{equation}
\begin{aligned}
    & \inf_{\bm v_h\in\mathbb H_h} \norm{(\mc A')^{-1}\bm w - \bm v_h}_{\mathbb H} \le \inf_{v_h \in W_h^1}\norm{\phi_1 - v_h}_{H^1(S)} + \inf_{v_h\in W_h^2}\norm{\phi_2 - v_h}_{H^{-\frac{1}{2}}_{\overline{S}}(\Gamma)} \\
    &\le \inf_{v_h \in V_h}\norm{\phi_1 - v_h}_{H^1(S)} + \inf_{v_h\in V_h}\norm{\phi_2 - v_h}_{ H^{-\frac{1}{2}}_{\overline{S}} (\Gamma)}  \\
    &\le \norm{\phi_1 - \pi_{L^2}\phi_1}_{H^1(S)} + \norm{\phi_2 - \pi_{L^2}\phi_2}_{H^{-\frac{1}{2}}_{\overline{S}} (\Gamma)} \le C h^\sigma \paren{  \norm{\phi_1}_{H^{1+\sigma}(S)} + \norm{\phi_2}_{H^{-\frac{1}{2}+\sigma}_{\overline{S}}(\Gamma)} } \\
    &\le C h^\sigma\paren{ \norm{w_1}_{H^{-\frac{1}{2}+\sigma}_{\overline{S}}(\Gamma)} + \norm{w_2}_{H^{\frac{1}{2}+\sigma}(S)} }  \le C h^\sigma \norm{\bm w}_Z.
\end{aligned}
\end{equation}
Therefore, we have
\begin{equation}
    |\inprod{\bm u - \bm u_h}{\bm w}_{L^2}| \le C \inf_{\bm v_h\in\mathbb H_h}\norm{\bm u - \bm u_h}_{\mathbb H} \norm{(\mc A')^{-1}\bm w - \bm v_h}_{\mathbb H} \le C h^\sigma \norm{\bm w}_{\rev{Z}}\norm{\bm u - \bm u_h}_{\mathbb H}.
\end{equation}
Moreover,
\begin{equation}
    \norm{U-U_h}_{L^2(S)} + \norm{\psi-\psi_h}_{H^{-\frac{1}{2}-\sigma}_{\overline{S}}(\Gamma)} \le \sup_{\bm w \neq 0}\frac{|\inprod{\bm u-\bm u_h}{\bm w}_{L^2}|}{\norm{\bm w}_{\rev{Z}}} \le C h^\sigma \norm{\bm u-\bm u_h}_{\mathbb H}.
\end{equation}
\end{proof}
Since the problem~\eqref{eqn:weak-eqn} satisfies the G\aa rding inequality~\eqref{eqn:garding}, the solvability and convergence of the finite element approximation~\eqref{eqn:Galerkin} are also guaranteed, and we have the following theorem.
\begin{theorem}
    Given $f\in H^{-1}_{\overline{S}}(\Gamma)$ and $g$ with $\norm{g}_{\partial S}<+\infty$, there exists $h_0> 0$ such that for any $h$, $0 < h < h_0$, the problem~\eqref{eqn:Galerkin} has a unique solution $\bm u_h = (U_h, \psi_h)\in \mathbb H_h$. Moreover, if $f\in H^{\frac{1}{2}+\varepsilon}_{\overline{S}}(\Gamma)$, then we have the error estimate
    \begin{equation}
        \norm{\bm u - \bm u_h}_{\mathbb H}  \le C h \paren{\norm{f}_{H^{\frac{1}{2}+\varepsilon}_{\overline{S}}(\Gamma)} + \norm{g}_{\partial S}}.
    \end{equation}
\end{theorem}
\begin{proof}
    We begin by deriving an estimate for any solution to~\eqref{eqn:Galerkin} that may exist. Note that we always have
    \begin{equation}
        a(\bm u - \bm u_h, \bm v_h) = 0, \quad \forall \bm v_h\in \mathbb H_h.
    \end{equation}
    For any $\sigma \in (0,\frac{1}{2})$, \rev{by} Lemma~\ref{lem:vs-decomp} and~\ref{lem:duality}, we have \rev{the following estimate},
    \begin{equation}
    \begin{aligned}
        &|\dual{\mc B(\bm u-\bm u_h)}{\bm u-\bm u_h}_{L^2(S)}| \\
        &\le |\dual{V_{S,1}(\psi-\psi_h)}{(\psi-\psi_h)}_{L^2(S)}| + |\inprod{U-U_h}{U-U_h}_{L^2(S)}| \\
        & \le \norm{V_{S,1}(\psi-\psi_h)}_{H^{\frac{1}{2}+\sigma}(S)} \norm{\psi-\psi_h}_{H^{-\frac{1}{2}-\sigma}_{\overline{S}}(\Gamma)}  + \norm{U-U_h}_{L^2(S)}^2 \\
        & \le C \paren{ \norm{\psi-\psi_h}_{H^{-\frac{1}{2}}_{\overline{S}}(\Gamma)}\norm{\psi-\psi_h}_{H^{-\frac{1}{2}-\sigma}_{\overline{S}}(\Gamma)}+ \norm{U-U_h}_{L^2(S)}^2 } \\
        &\le C h^{\sigma} \norm{\bm u-\bm u_h}_{\mathbb H}^2.
    \end{aligned}
    \end{equation}
    For any $\bm v_h\in \mathbb H_h$, by the continuity of $a(\cdot,\cdot)$ and the G\aa rding inequality~\eqref{eqn:garding},
    \begin{equation}
    \begin{aligned}
        & C\norm{\bm u - \bm u_h}_{\mathbb H}^2 \le a(\bm u-\bm u_h, \bm u-\bm u_h) + \dual{\mc B(\bm u-\bm u_h)}{\bm u-\bm u_h}_{L^2(S)} \\
        &= a(\bm u-\bm u_h, \bm u-\bm v_h) + \dual{\mc B(\bm u-\bm u_h)}{\bm u-\bm u_h}_{L^2(S)} \\
        &\le C \norm{\bm u-\bm u_h}_{\mathbb H} \norm{\bm u-\bm v_h}_{\mathbb H} + C h^{\sigma} \norm{\bm u-\bm u_h}_{\mathbb H}^2.
    \end{aligned}
    \end{equation}
    Since $\sigma > 0$, there exists $h_0>0$ such that for every $h \in (0,h_0)$, we have
    \begin{equation}
        \norm{\bm u-\bm u_h}_{\mathbb H} \le C \norm{\bm u-\bm v_h}_{\mathbb H}, \quad \forall \bm v_h \in \mathbb H_h.
    \end{equation}
    Then, applying Lemma~\ref{thm:interp}, we have
    \begin{equation}\label{eqn:error-bound}
    \begin{aligned}
        &\norm{\bm u-\bm u_h}_{\mathbb H} \le \inf_{\bm v_h\in\mathbb H_h} \norm{\bm u-\bm v_h}_{\mathbb H} \le C h \paren{\norm{U_r}_{H^{2}(S)} + \norm{\psi_r}_{H^{\frac{1}{2}}_{\overline{S}}(\Gamma)}  } \\
        &\le C h \paren{\norm{f}_{H^{\frac{1}{2}+\varepsilon}_{\overline{S}}(\Gamma)} + \norm{g}_{\partial S}}.
    \end{aligned}
    \end{equation}

We now discuss the existence and uniqueness of the solution to~\eqref{eqn:Galerkin}. Since $\mathbb H_h\subset \mathbb H$ is finite-dimensional, the system~\eqref{eqn:Galerkin} has the same number of equations and unknowns. Therefore, it suffices to consider the homogeneous problem. For $f\equiv 0$ and $g\equiv 0$, we have $\bm u \equiv 0$ by Theorem~\ref{thm:solvability}. For sufficiently small $h$, \eqref{eqn:error-bound} also implies $\bm u_h\equiv 0$. This shows that the finite element approximation~\eqref{eqn:Galerkin} has a unique solution $\bm u_h=(U_h, \psi_h)$.
\end{proof}

\subsection{Numerical examples}
We present numerical examples to illustrate the effectiveness of the enriched finite element method with singular elements. We use the standard piecewise linear finite element method, referred to as the ``standard FEM,'' as a baseline for comparison with the proposed method.
Since the exact solution is generally unavailable, we estimate the numerical error in the energy norm through successive grid refinement. Let $\bm u_h$ and $\bm u_{2h}$ denote the discrete solutions on meshes of size $h$ and $2h$, respectively. For numerical convenience, we compute $|\bm u_h - \bm u_{2h}|_a$, where the energy norm is defined by $|\cdot|_a \coloneqq \sqrt{a(\cdot,\cdot)}$, which is more straightforward to implement than the $H^{-\frac{1}{2}}$ norm.

\noindent\textbf{Example 1.}  
Let the open boundary be a straight line segment given by $ S = \{(x,0) : x \in (-1,1)\} $.  
We set the right-hand side to the constant function $ f(x,y) = 1 $ and impose the Neumann condition $ \bm{\nu} \cdot \nabla_\Gamma U = -1 $ at the two endpoints.  
Table~\ref{tab:fem_errors} reports the errors $ \|\bm{u}_h - \bm{u}_{2h}\|_a $ and the convergence orders for both the standard FEM and the enriched FEM. We observe that the standard FEM achieves only half-order convergence, while the enriched FEM achieves first-order convergence.
\begin{table}[htbp]
\centering
\caption{Example 1: energy‐norm error and convergence order for standard and enriched FEM}
\label{tab:fem_errors}
\begin{tabular}{r cc cc}
\hline
  $N$ & \multicolumn{2}{c}{Standard FEM} & \multicolumn{2}{c}{Enriched FEM} \\
      & $\norm{ \bm u_h - \bm u_{2h} }_a$ & Order       & $\norm{ \bm u_h - \bm u_{2h} }_a$ & Order       \\
\hline
   64 & 1.17E-01 & --- & 2.71E-02 & --- \\
  128 & 8.12E-02 & 0.52 & 1.45E-02 & 0.90 \\
  256 & 5.68E-02 & 0.52 & 7.65E-03 & 0.92 \\
  512 & 3.99E-02 & 0.51 & 4.02E-03 & 0.93 \\
 1024 & 2.81E-02 & 0.51 & 2.10E-03 & 0.94 \\
\hline
\end{tabular}
\end{table}
Figure~\ref{fig:c1-sol} shows the numerical solutions $ U $ and $ \psi $ plotted against the arc-length parameter $ s \in (-1,1) $ on a grid with $ N = 64 $, obtained using both methods. Since the singular element is unbounded, we plot the enriched FEM solution only over the interval $ s \in [-1 + 10^{-3},\, 1 - 10^{-3}] $. From the figures, we observe that the numerical solutions essentially agree, except that the standard FEM exhibits spurious oscillations near the two endpoints, whereas the enriched FEM successfully captures the singular behavior of the solution without such oscillations.  
The solution $ u $ is also plotted in Figure~\ref{fig:uplot-a}.
\begin{figure}[htbp]
    \centering
    \includegraphics[width=0.4\linewidth]{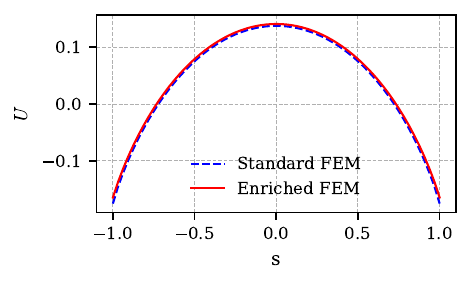}
    \includegraphics[width=0.4\linewidth]{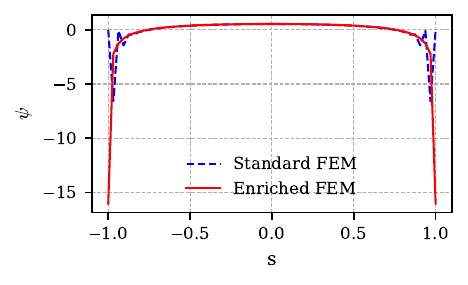}
    \caption{Example 1: numerical solutions of $U$ (left) and $\psi$ (right) computed with $N=64$.}
    \label{fig:c1-sol}
\end{figure}

\noindent\textbf{Example 2.}  
Let the open boundary be a semicircle defined as $ S = \{(\cos s, \sin s) : s \in (0, \pi)\} $.  
We set the right-hand side as a smooth function $ f(x,y) = \sin(\pi x) $ and impose the Neumann condition $ \bm{\nu} \cdot \nabla_\Gamma U = -\frac{1}{2} \int_S f\,ds $ at the two endpoints to satisfy the far-field condition for $ u $.
Table~\ref{tab:fem_errors2} shows the corresponding energy-norm errors and convergence rates for both the standard FEM and the enriched FEM. Again, we observe that the standard FEM achieves only half-order convergence, while the enriched FEM attains first-order convergence.
\begin{table}[htbp]
\centering
\caption{Example 2: energy‐norm error and convergence order for standard and enriched FEM}
\label{tab:fem_errors2}
\begin{tabular}{r cc cc}
\hline
  $N$ & \multicolumn{2}{c}{Standard FEM} & \multicolumn{2}{c}{Enriched FEM} \\
      & $\norm{ \bm u_h - \bm u_{2h} }_a$ & Order       & $\norm{ \bm u_h - \bm u_{2h} }_a$ & Order       \\
\hline
   64 & 2.53E-02 & ---  & 1.44E-02 & --- \\
  128 & 1.60E-02 & 0.66 & 7.27E-03 & 0.99 \\
  256 & 1.06E-02 & 0.60 & 3.67E-03 & 0.99 \\
  512 & 7.20E-03 & 0.56 & 1.85E-03 & 0.99 \\
 1024 & 4.98E-03 & 0.53 & 9.36E-04 & 0.99 \\
\hline
\end{tabular}
\end{table}
Figure~\ref{fig:c2-sol} shows the numerical solutions $ U $ and $ \psi $ plotted against the arc-length parameter $ s \in (-\pi/2, \pi/2) $ on a grid with $ N = 64 $, computed using both methods. We plot the enriched FEM solution only for $ s \in [-\pi/2 + 10^{-3},\, \pi/2 - 10^{-3}] $. As before, we observe that the standard FEM produces spurious oscillations near the endpoints, while the enriched FEM accurately captures the singular behavior without oscillations.  
\begin{figure}[htbp]
    \centering
    \includegraphics[width=0.4\linewidth]{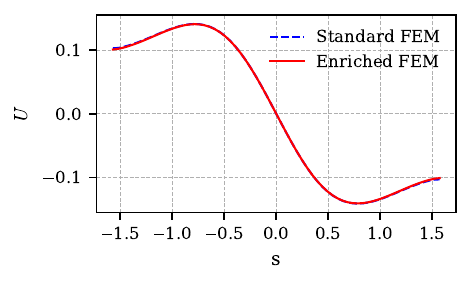}
    \includegraphics[width=0.4\linewidth]{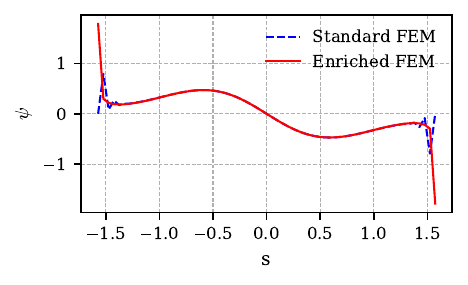}
    \caption{Example 2: numerical solutions of $U$ (left) and $\psi$ (right) computed with $N=64$. }
    \label{fig:c2-sol}
\end{figure}
The solution $ u $ outside the open boundary is shown in Figure~\ref{fig:uplot-b}.
\begin{figure}[htbp]
    \centering
    \begin{subfigure}[b]{0.4\linewidth}
        \centering
        \includegraphics[width=\linewidth]{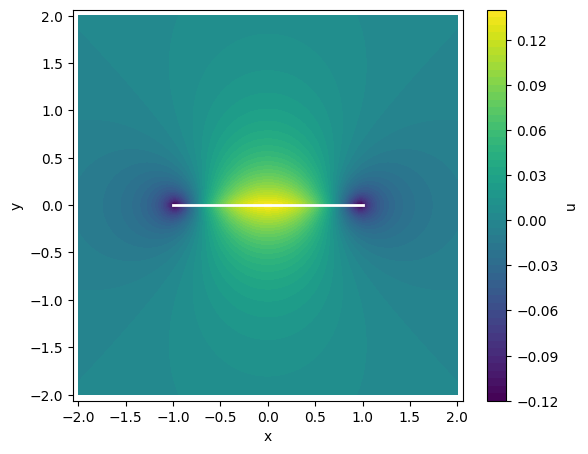}
        \caption{Straight line segment}
        \label{fig:uplot-a}
    \end{subfigure}
    \quad
    \begin{subfigure}[b]{0.4\linewidth}
        \centering
        \includegraphics[width=\linewidth]{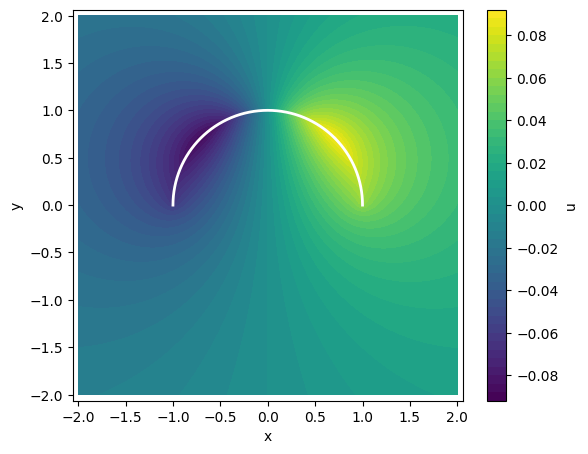}
        \caption{Semicircular arc}
        \label{fig:uplot-b}
    \end{subfigure}
    \caption{Numerical solutions of $u$ for two different configurations.}
    \label{fig:uplot}
\end{figure}

\section{\texorpdfstring{\rev{Conclusion}}{Conclusion}}\label{sec:conclusion}
\rev{We studied a bulk--surface coupled elliptic system posed on an open boundary. By representing the bulk field through a single-layer potential, the original problem in \(\mathbb R^2\setminus\overline S\) was reduced to an integro-differential system on the open arc. This formulation led to existence and uniqueness of weak solutions, regularity estimates, and the leading endpoint singularities of both the surface unknown \(U\) and the single-layer density \(\psi\).}

\rev{The singular expansion also guided the finite element discretization. Instead of refining the mesh near the endpoints, we enriched a quasi-uniform finite element space with the singular functions predicted by the analysis. The resulting method incorporates the endpoint behavior directly into the approximation space: the error analysis and numerical experiments show first-order convergence for the enriched method, while the standard piecewise linear method is limited by the endpoint singularity.}

\appendix

\section{Proofs of some lemmas}

\subsection{Proof of Lemma~\ref{lem:hso-norm}}\label{pf-hso-norm}
The inclusion $H^s_{\overline{\Omega}}(\R^n)\subset H^s_{\overline{\Omega}}(O)$ holds since
\begin{equation}
    \norm{f}_{H^s(O)} = \inf_{g\in H^s(\R^n),\at{g}{O} = f}\norm{g}_{H^s(\R^n)} \le \norm{f}_{H^s(\R^n)}, 
\end{equation}
for every $f\in C_c^\infty(\Omega)$. 
We choose $\vph\in C_c^\infty(O)$ such that $\vph = 1$ in $\Omega$ and $\vph = 0$ in $\R^d\setminus\overline{O}$, and there exists a constant $C$ such that for a multi-index $\alpha$,
\begin{equation}
    |\partial^{\alpha} \vph| \le C,
\end{equation}
where the constant $C$ only depends on $\Omega,O,\alpha$.
For an integer $k\ge 0$ and $f\in C_c^\infty(\Omega)$, we have
\begin{equation}
\begin{aligned}
    &\norm{f}_{H^k(\R^n)}^2 = \norm{\vph f}_{H^k(\R^n)}^2 = \sum_{|\alpha| \le k} \norm{\partial^\alpha(\vph f)}_{L^2(\R^n)}^2 \\
    & \le C\sum_{|\alpha| \le k}\sum_{|\beta|+|\gamma| = |\alpha|} \norm{\partial^\beta\vph \partial^\gamma f}_{L^2(O)}^2 \le C\sum_{|\alpha| \le k}\sum_{|\beta|+|\gamma| = |\alpha|} \norm{\partial^\beta\vph}_{L^\infty(O)}^2  \norm{\partial^\gamma f}_{L^2(O)}^2 \\  
    & \le C\sum_{|\alpha| \le k} \norm{\partial^\alpha f}_{L^2(O)}^2 = C \norm{f}_{H^k(O)}^2,
\end{aligned}
\end{equation}
where the constant $C$ only depends on $\Omega, O, k$.
By the density argument, we have $ H^s_{\overline{\Omega}}(O) = H^s_{\overline{\Omega}}(\R^n)$.

\subsection{Proof of Lemma~\ref{lem:norm-equv-1}}\label{sec:pf-lem-1}
Since $\bm{X}$ is a regular parameterization, there exists $\gamma > 0$ such that for all $t\in\mathbb T$,
\begin{equation}
    \gamma\le \abs{\bm{X}'(t)} \le \norm{\bm{X}}_{C^1(\T)}.
\end{equation}
This implies $|\bm{X}'(t)|^{-1}\in C^{k-1,\alpha}(\T)$.
For $s = 0$, since
\begin{equation}
\norm{f}_{L^2(\Gamma)}^2 = \int_\Gamma| f( x )|^2\,ds_{ x } = \int_{\mathbb T} |(f\circ\bm{X})(t)|^2 \abs{\bm{X}'(t)}\,dt,
\end{equation} 
we have
\begin{equation}
\gamma \norm{f\circ\bm{X}}_{L^2(\mathbb T)}^2 \le \norm{f}_{L^2(\Gamma)}^2  \le  \norm{\bm{X}}_{C^1} \norm{f\circ\bm{X}}_{L^2(\mathbb T)}^2.
\end{equation}

Now we consider the case when $s=m$ is an integer with $1\le m\le k$. 
For any integer $j\ge 1$ and functions $g$ and $f$, it holds that
\begin{equation}\label{eqn:bell-poly}
(g(t)\partial_t)^{j} f(t) = \sum_{m=1}^{j}
g(t)^{m}
B_{j,m}\!\left(
g^{(1)}(t),g^{(2)}(t),\ldots,g^{(j-m+1)}(t)
\right)
f^{(m)}(t),
\end{equation}
where the $B_{j,m}$ are polynomials.
For any integer $1\le \beta \le m$ and $t\in\mathbb T$, using the formula~\eqref{eqn:bell-poly}, we have,
\begin{equation}
\abs{\paren{(\partial^\beta f)\circ \bm{X} }(t)} = \abs{\paren{\abs{\bm{X}'(t)}^{-1}\partial_t}^\beta \paren{f\circ\bm{X}}(t)}  \le  C \sum_{j = 0}^\beta \abs{\partial_t^j (f\circ\bm{X})(t)}, 
\end{equation}
where the constant $C$ depends on $\bm{X}$ but not on $f$.
Consequently, 
\begin{equation}
\begin{aligned}
    \norm{\partial^\beta f}_{L^2(\Gamma)}^2 \le  C \sum_{j = 0}^\beta  \norm{\partial_t^j(f\circ \bm{X})}_{L^2(\mathbb T)}^2 
    & \le C \norm{f\circ \bm{X}}_{H^m(\mathbb T)}^2.        
\end{aligned}
\end{equation}
This shows 
\begin{equation}
    \norm{f}_{H^m(\Gamma)} \le C \norm{f\circ \bm{X}}_{H^m(\mathbb T)}.
\end{equation}
For the other side of the inequality, we write
\begin{equation}
    \paren{(\partial^\beta f)\circ \bm{X} }(t) = \abs{\bm{X}'(t)}^{-\beta}\partial_t^\beta \paren{f\circ\bm{X}} (t)+ \mc R[\bm{X}](t),
\end{equation}
where $\mc R[\bm{X}]$ involves derivatives of $f\circ\bm{X}$ up to the $(\beta-1)$-th order. 
For every $t\in \T$,
\begin{equation}
    \abs{\partial_t^\beta \paren{f\circ\bm{X}} (t)} \le \gamma^{\beta} \paren{\abs{\paren{(\partial^\beta f)\circ \bm{X} }(t)  |+|\mc R[\bm{X}] (t)} }.
\end{equation}
By induction, we have
\begin{equation}
    \norm{\partial_t^{\beta}\paren{f\circ\bm{X}}}_{L^2(\mathbb T)}  \le  C\paren{ \norm{\partial^\beta f}_{L^2(\Gamma)} + \norm{f}_{H^{m-1}(\Gamma)} }  \le  C \norm{f}_{H^m(\Gamma)}.
\end{equation}
This gives
\begin{equation}
    \norm{f}_{H^m(\Gamma)} \ge C\norm{f\circ \bm{X}}_{H^m(\mathbb T)}.
\end{equation}
By \rev{Sobolev }interpolation and duality\rev{ \cite[Appendix~B]{McLean2000}}, the result holds for all $|s|\le k$.

\subsection{Proof of Lemma~\ref{lem:vg-decomp}}\label{pf-vg-decom}
We start with some calculus results.
For $t,s,s+h\in \T$, we define
\begin{align}
    &\Delta f = f(s) - f(t), \\
    &T_{s,h} f(s,t) = f(s+h, t),\\ 
    &\Delta_{s,h} f(s,t) = f(s+h, t) - f(s,t) .
\end{align}
\begin{lemma}\label{lem:diff-est}
Let $\bm{X}\in C^{1,\delta}(\T)$. It holds that,
\begin{align}
    &|\Delta \bm{X}| \le \norm{\bm{X}}_{C^1} |s-t|, \\
    &|T_{s,h} \Delta \bm{X}| \le \norm{\bm{X}}_{C^1} |s+h-t|, \\
    &|\bm{X}'(t) - \frac{\Delta \bm{X}}{s-t}| \le \norm{\bm{X}}_{C^{1,\delta}} |s-t|^\delta, \\
    &|\bm{X}'(s) - \frac{\Delta \bm{X}}{s-t}| \le \norm{\bm{X}}_{C^{1,\delta}} |s-t|^\delta, \\
    &|T_{s,h}\Delta \bm{X} - \Delta \bm{X}| \le \norm{\bm{X}}_{C^1} |h|.
\end{align}
\end{lemma}

\begin{lemma}\label{lem:knl-est}
Let $\bm{X}\in C^{1,\delta}$ for $\delta\in(0,1)$ with $\starnorm{\bm{X}}>0$. We define
\begin{align}
    K_1(s,t) &= \frac{\Delta \bm{X} \cdot \paren{\bm{X}'(t) - \frac{\Delta \bm{X}}{s-t}}}{|\Delta \bm{X}|^2}, \label{eqn:K1-def}\\
    Q_1(s,t) &= \frac{\Delta \bm{X}\cdot\paren{ \bm{X}'(s) - \frac{\Delta \bm{X}}{s-t}}}{|\Delta \bm{X}|^2}. \label{eqn:Q1-def}
\end{align}
Then, it holds that
\begin{align}
    &|K_1(s,t)|, |Q_1(s,t)| \le \frac{\norm{\bm{X}}_{C^{1,\delta}}}{\starnorm{\bm{X}}} |s-t|^{\delta-1}, \\
    &|\Delta_{s,h}K_1(s,t)|, |\Delta_{s,h}Q_1(s,t)|\le C\frac{\norm{\bm{X}}_{C^{1,\delta}}^2}{\starnorm{\bm{X}}^2} \paren{ |s+h-t|^{-1}|h|^\delta +  |h| |s+h-t|^{\delta-2} },
\end{align}
\end{lemma}

\begin{proof}    
\begin{equation}
    \begin{aligned}
        &\Delta_{s,h} K_1(s,t) \\
        &= \frac{T_{s,h}\Delta \bm{X} \cdot \paren{\bm{X}'(t) - \frac{T_{s,h}\Delta \bm{X}}{s+h-t}}}{|T_{s,h}\Delta \bm{X}|^2} - \frac{\Delta \bm{X} \cdot \paren{\bm{X}'(t) - \frac{\Delta \bm{X}}{s-t}}}{|\Delta \bm{X}|^2} \\
        &= \frac{T_{s,h}\Delta \bm{X} \cdot \paren{\bm{X}'(t) - \frac{T_{s,h}\Delta \bm{X}}{s+h-t}} |\Delta \bm{X}|^2- \Delta \bm{X} \cdot \paren{\bm{X}'(t) - \frac{\Delta \bm{X}}{s-t}}|T_{s,h}\Delta \bm{X}|^2}{|\Delta \bm{X}|^2|T_{s,h}\Delta \bm{X}|^2} \\
        &= \frac{A_1 + A_2 + A_3 + A_4}{|\Delta \bm{X}|^2|T_{s,h}\Delta \bm{X}|^2},
    \end{aligned}
\end{equation}
where
\begin{align}
    &A_1 =T_{s,h}\Delta \bm{X} \cdot \paren{\bm{X}'(t) - \frac{T_{s,h}\Delta \bm{X}}{s+h-t}} |\Delta \bm{X}|^2 - T_{s,h}\Delta \bm{X} \cdot \paren{\bm{X}'(t) - \frac{\Delta \bm{X}}{s-t}} |\Delta \bm{X}|^2 \\
    &A_2 = T_{s,h}\Delta \bm{X} \cdot \paren{\bm{X}'(t) - \frac{\Delta \bm{X}}{s-t}} |\Delta \bm{X}|^2 - \Delta \bm{X} \cdot \paren{\bm{X}'(t) - \frac{\Delta \bm{X}}{s-t}} |\Delta \bm{X}|^2 \\
    &A_3 = \Delta \bm{X} \cdot \paren{\bm{X}'(t) - \frac{\Delta \bm{X}}{s-t}} |\Delta \bm{X}|^2 - \Delta \bm{X} \cdot \paren{\bm{X}'(t) - \frac{\Delta \bm{X}}{s-t}} (\Delta \bm{X}\cdot T_{s,h}\Delta \bm{X}) \\
    &A_4 = \Delta \bm{X} \cdot \paren{\bm{X}'(t) - \frac{\Delta \bm{X}}{s-t}} (\Delta \bm{X}\cdot T_{s,h}\Delta \bm{X}) - \Delta \bm{X} \cdot \paren{\bm{X}'(t) - \frac{\Delta \bm{X}}{s-t}}|T_{s,h}\Delta \bm{X}|^2.
\end{align}
Applying Lemma~\ref{lem:diff-est}, they satisfy the following bounds,
\begin{align}
    &|A_1|\le |\Delta \bm{X}|^2|T_{s,h}\Delta \bm{X}| \norm{\bm{X}}_{C^{1,\delta}} |h|^\delta, \\
    &|A_2| ,|A_3|, |A_4|\le |\Delta \bm{X}|^2 \norm{\bm{X}}_{C^{1,\delta}}^2|h| |s-t|^\delta.
\end{align}
We also have,
\begin{equation}
    \begin{aligned}
        &\Delta_{s,h} Q_1(s,t) \\
        &= \frac{T_{s,h}\Delta \bm{X} \cdot \paren{T_{s,h}\bm{X}'(s) - \frac{T_{s,h}\Delta \bm{X}}{s+h-t}}}{|T_{s,h}\Delta \bm{X}|^2} - \frac{\Delta \bm{X} \cdot \paren{\bm{X}'(s) - \frac{\Delta \bm{X}}{s-t}}}{|\Delta \bm{X}|^2} \\
        &= \frac{T_{s,h}\Delta \bm{X} \cdot \paren{T_{s,h}\bm{X}'(s) - \frac{T_{s,h}\Delta \bm{X}}{s+h-t}} |\Delta \bm{X}|^2- \Delta \bm{X} \cdot \paren{\bm{X}'(s) - \frac{\Delta \bm{X}}{s-t}}|T_{s,h}\Delta \bm{X}|^2}{|\Delta \bm{X}|^2|T_{s,h}\Delta \bm{X}|^2} \\
        &= \frac{B_1 + B_2 + B_3 + B_4}{|\Delta \bm{X}|^2|T_{s,h}\Delta \bm{X}|^2}
    \end{aligned}
\end{equation}
where
\begin{align}
    B_1 &=T_{s,h}\Delta \bm{X} \cdot \paren{T_{s,h}\bm{X}'(s) - \frac{T_{s,h}\Delta \bm{X}}{s+h-t}} |\Delta \bm{X}|^2 - T_{s,h}\Delta \bm{X} \cdot \paren{\bm{X}'(s) - \frac{\Delta \bm{X}}{s-t}} |\Delta \bm{X}|^2 \\
    B_2 &= T_{s,h}\Delta \bm{X} \cdot \paren{\bm{X}'(s) - \frac{\Delta \bm{X}}{s-t}} |\Delta \bm{X}|^2 - \Delta \bm{X} \cdot \paren{\bm{X}'(s) - \frac{\Delta \bm{X}}{s-t}} |\Delta \bm{X}|^2 \\
    B_3  &= \Delta \bm{X} \cdot \paren{\bm{X}'(s) - \frac{\Delta \bm{X}}{s-t}} |\Delta \bm{X}|^2 - \Delta \bm{X} \cdot \paren{\bm{X}'(s) - \frac{\Delta \bm{X}}{s-t}} (\Delta \bm{X}\cdot T_{s,h}\Delta \bm{X}) \\
    B_4  &= \Delta \bm{X} \cdot \paren{\bm{X}'(s) - \frac{\Delta \bm{X}}{s-t}} (\Delta \bm{X}\cdot T_{s,h}\Delta \bm{X}) - \Delta \bm{X} \cdot \paren{\bm{X}'(s) - \frac{\Delta \bm{X}}{s-t}}|T_{s,h}\Delta \bm{X}|^2.
\end{align}
Similarly, we have
\begin{align}
    &|B_1|\le 2|\Delta \bm{X}|^2|T_{s,h}\Delta \bm{X}| \norm{\bm{X}}_{C^{1,\delta}} |h|^\delta, \\
    &|B_2| ,|B_3|, |B_4|\le |\Delta \bm{X}|^2 \norm{\bm{X}}_{C^{1,\delta}}^2|h| |s-t|^\delta.
\end{align}
Applying Lemma~\ref{lem:diff-est}, we obtain the results.


\end{proof}

\begin{lemma}\label{lem:L2-op}
Let $\bm{X}\in C^{1,\delta}(\T)$ with $\delta\in(0,1)$ and $\starnorm{\bm{X}}>0$. 
Define the operator
\begin{equation}
    \mc L[u] (s) = \int_\T L(s,t) u(t)\,dt,\quad L(s,t) = \log|\bm{X}(s)-\bm{X}(t)| - \log\abs{2\sin\frac{s-t}{2}}.
\end{equation}
Then, the operator $\mc L: H^{-1+s}(\T) \to H^{s+\varepsilon}(\T)$ is bounded for $s \in [0,1], \varepsilon \in (0,\delta)$, with
\begin{equation}
\begin{aligned}
    \norm{\mc L[u]}_{H^{s+\varepsilon}}  \le  C \frac{\norm{\bm{X}}_{C^{1,\delta}}^2}{\starnorm{\bm{X}}^2} \norm{u}_{H^{-1+s}}
\end{aligned}
\end{equation}
\end{lemma}
\begin{proof}

We first consider the following two operators
\begin{align}
     \mc K[u](s) &= \int_\T K(s, t) u(t) \,dt, \\
     \mc Q[u](s) &= \int_\T Q(s, t) v(t) \,dt.
\end{align}
where the kernels are defined as
\begin{align}
    &K(s,t) \coloneqq -\partial_t L(s,t) =  \frac{(\bm{X}(s) - \bm{X}(t))\cdot \bm{X}'(t)}{ (\bm{X}(s)-\bm{X}(t))^2 } - \frac{1}{2} \cot \paren{\frac{s-t}{2}}  ,\\
    &Q(s,t)\coloneqq\partial_s L(s,t) = \frac{(\bm{X}(s) - \bm{X}(t))\cdot \bm{X}'(s)}{ (\bm{X}(s)-\bm{X}(t))^2 } - \frac{1}{2} \cot \paren{\frac{s-t}{2} }.
\end{align}
We perform the splitting for the two kernels,
\begin{align}
    K(s,t) = K_1(s,t)  + R(s,t),\quad  Q(s,t)  = Q_1(s,t) + R(s,t),
\end{align}
where $K_1$ and $Q_1$ are defined in~\eqref{eqn:K1-def} and~\eqref{eqn:Q1-def}, and $R$ is given by
\begin{align}
    R(s,t) = \frac{1}{s-t} - \frac{1}{2}\cot\paren{\frac{s-t}{2}}.
\end{align}
Simple calculation shows that
\begin{equation}
    |R(s,t)| \le C|s-t|, \quad \Delta_{s,h}R(s,t) \le C |h|(|s-t|^2+|s+h-t|^2).
\end{equation}
Using Lemma~\ref{lem:knl-est}, we obtain that,
\begin{align}
    M_1 \coloneqq  \sup_s \int_\T |K(s,t)| \,dt \le  \sup_s \int_{-\pi}^\pi |s-t|^{\delta-1} +|s-t|\,dt = C\frac{\norm{\bm{X}}_{C^{1,\delta}}}{\starnorm{\bm{X}}} , \\
    M_2 \coloneqq  \sup_t \int_\T |K(s,t)| \,ds \le  \sup_t \int_{-\pi}^\pi |s-t|^{\delta-1} +|s-t|\,ds = C\frac{\norm{\bm{X}}_{C^{1,\delta}}}{\starnorm{\bm{X}}} .   
\end{align}
Applying Schur's test leads to
\begin{equation}
\begin{aligned}
    \norm{\mc K[u]}_{L^2}^2 \le \int_\T \paren{\int_\T K(s,t) u(t)\,dt}^2\,ds  \le  M_1 M_2 \norm{u}_{L^2}^2 \le   C\frac{\norm{\bm{X}}_{C^{1,\delta}}^2}{\starnorm{\bm{X}}^2} \norm{u}_{L^2}^2,
\end{aligned}
\end{equation}
Similarly,
\begin{equation}
    \norm{\mc Q[u]}_{L^2}^2 \le \int_\T \paren{\int_\T Q(s,t) u(t)\,dt}^2\,ds  \le  M_1 M_2 \norm{u}_{L^2}^2 \le C\frac{\norm{\bm{X}}_{C^{1,\delta}}^2}{\starnorm{\bm{X}}^2} \norm{u}_{L^2}^2.
\end{equation} 
Furthermore, we have 
\begin{equation}
\begin{aligned}
    &\int_\T |\Delta_{s,h} K(s,t)|\,ds = \int_{t-2h}^{t-2h+2\pi}|\Delta_{s,h} K(s,t)|\,ds\\
    &= \int_{t-2h}^{t+h} |\Delta_{s,h} K(s,t)|\,ds + \int_{t+h}^{t-2h+2\pi}|\Delta_{s,h} K(s,t)|\,ds \\
    &\le C \frac{\norm{\bm{X}}_{C^{1,\delta}}^2}{\starnorm{\bm{X}}^2} \Big ( \int_{t-2h}^{t+h} |s-t|^{\delta-1} + |s+h-t|^{\delta-1}\,ds  \\
    &\,\,  + \int_{t+h}^{t-2h+2\pi}|s+h-t|^{-1}|h|^\delta + |h| |s+h-t|^{\delta-2}\,ds  +|h|  \Big)\\
    &\le C \frac{\norm{\bm{X}}_{C^{1,\delta}}^2}{\starnorm{\bm{X}}^2}  |h|^\delta \log|h|.
\end{aligned}
\end{equation}
and,
\begin{equation}
\begin{aligned}
    &\int_\T |\Delta_{s,h} K(s,t)|\,dt = \int_{s-h}^{s-h+2\pi}|\Delta_{s,h} K(s,t)|\,dt\\
    &= \int_{s-h}^{s+2h}|\Delta_{s,h} K(s,t)|\,dt + \int_{s+2h}^{s-h+2\pi}|\Delta_{s,h} K(s,t)|\,dt\\
    &\le C \frac{\norm{\bm{X}}_{C^{1,\delta}}^2}{\starnorm{\bm{X}}^2}\Big( \int_{s-h}^{s+2h} |s-t|^{\delta-1} + |s+h-t|^{\delta-1}\,ds  \\
    &\,\,  + \int_{s+2h}^{s-h+2\pi}  |s+h-t|^{-1}|h|^\delta +   |h| |s+h-t|^{\delta-2}\,ds   +|h|\Big)\\
    &\le C \frac{\norm{\bm{X}}_{C^{1,\delta}}^2}{\starnorm{\bm{X}}^2}  |h|^\delta \log|h|.
\end{aligned}
\end{equation}
Here, we have used $\starnorm{\bm{X}} \le \norm{\bm{X}}_{C^1}\le \norm{\bm{X}}_{C^{1,\delta}}$.
Hence, it holds that
\begin{align}
    M_{1,h} &\coloneqq \sup_{t}\norm{\Delta_{s,h}K(\cdot, t)}_{L^1} \le C \frac{\norm{\bm{X}}_{C^{1,\delta}}^2}{\starnorm{\bm{X}}^2}  |h|^\delta \log|h|, \\ 
    M_{2,h} &\coloneqq \sup_{s}\norm{\Delta_{s,h}K(s, \cdot)}_{L^1} \le C \frac{\norm{\bm{X}}_{C^{1,\delta}}^2}{\starnorm{\bm{X}}^2}  |h|^\delta \log|h|.
\end{align}
Since $\varepsilon <\delta$, by applying Schur's test again, we obtain
\begin{equation}
\begin{aligned}
    [\mc K[u]]_{H^\varepsilon}^2 &= \int_{\T} \int_{\T} \frac{|\mc K[u](x) - \mc K[u](y)|^2}{|x-y|^{1+2\varepsilon}} \,dxdy  =  \int_{\T} \int_{\T} \frac{|\mc K[u](y+h) - \mc K[u](y)|^2}{|h|^{1+2\varepsilon}} \,dydh \\
    &=  \int_{\T} \frac{ \norm{  \int_\T \Delta_{y,h} K(y,t) u (t) \,dt  }_{L^2_y}^2}{|h|^{1+2\varepsilon}}\,dh  \le \int_\T \frac{M_{1,h} M_{2,h}\norm{u}_{L^2}^2}{|h|^{1 + 2\varepsilon}} \,dh\\
    &\le C \frac{\norm{\bm{X}}_{C^{1,\delta}}^4}{\starnorm{\bm{X}}^4}  \norm{u}_{L^2}^2\int_\T |h|^{-1-2\varepsilon+2\delta}\log|h|\,dh  \le C \frac{\norm{\bm{X}}_{C^{1,\delta}}^4}{\starnorm{\bm{X}}^4}  \norm{u}_{L^2}^2,
\end{aligned}
\end{equation}
Similarly,
\begin{equation}
    [\mc Q[u]]_{H^\varepsilon}^2 \le C \frac{\norm{\bm{X}}_{C^{1,\delta}}^4}{\starnorm{\bm{X}}^4}  \norm{u}_{L^2}^2.
\end{equation}
Therefore, we obtain
\begin{equation}\label{eqn:KQ_Heps_est}
    \norm{\mc K[u]}_{H^\varepsilon} \le C \frac{\norm{\bm{X}}_{C^{1,\delta}}^2}{\starnorm{\bm{X}}^2}  \norm{u}_{L^2}, \quad \norm{\mc Q[u]}_{H^\varepsilon} \le C \frac{\norm{\bm{X}}_{C^{1,\delta}}^2}{\starnorm{\bm{X}}^2}  \norm{u}_{L^2}.
\end{equation}

Now, we consider the map $\mc L:H^{-1}(\R)\to H^\varepsilon(\T)$.
For $u\in H^{-1}(\T) $, there exists $ v\in L^2(\T)  $ such that $u$ can be written as
\begin{equation}
    u = \overline{u} +\partial_s v,
\end{equation}
with $\overline{u} = \frac{1}{2\pi}\int_\T u(t) \,dt$ and $\int_\T v(t)\,dt = 0.$
Since $u\in H^{-1}(\T) $ and $1\in H^1(\T) $, 
\begin{align}
    &|\overline{u}| = \frac{1}{2\pi}|\dual{u}{1}_{L^2}| \le  \norm{u}_{H^{-1}},\\
    &\norm{v}_{L^2}^2 = \sum_{k\neq 0} |\wh v_k|^2 = \sum_{k\neq 0} \abs{\frac{\wh u_k}{ik}}^2 = \sum_{k\neq 0} \frac{|\wh u_k|^2}{|k|^2}  \le  2 \sum_{k\in\Z} \frac{|\wh u_k|^2}{1 + |k|^2} = 2\norm{u}_{H^{-1}}^2.
\end{align}
We can rewrite $\mc L[u]$ as
\begin{equation}
    \mc L[u](s) = \int_\T L(s,t)\paren{\overline{u} +\partial_t v(t)} \,dt = \overline{u}  \int_{\T} L(s,t) \,dt + \int_{\T} K(s,t) v(t)\,dt.
\end{equation}
Since the kernel $L(s,t)$ is bounded pointwise,
\begin{equation}
    |L(s,t)| \le  \abs{\log\frac{|\bm{X}(s)-\bm{X}(t)|}{|s-t|} - \log\frac{|2\sin\paren{(s-t)/2}}{|s-t|} } \le \norm{\bm{X}}_{C^1}+\frac{1}{\starnorm{\bm{X}} } + \log(\pi/2),
\end{equation}
for any $w\in L^1(\T)$, it holds that
\begin{equation}\label{eqn:L-l2est}
    \norm{\mc L[w]}_{L^2} \le C\paren{\norm{\bm{X}}_{C^1}+\frac{1}{\starnorm{\bm{X}} } + \log(\pi/2)} \norm{w}_{L^1}.
\end{equation}
By using the relation $\mc Q[u] = \partial_s \mc L[u]$ and~\eqref{eqn:KQ_Heps_est}, we obtain
\begin{equation}
\begin{aligned}
    &\norm{\mc L[u]}_{H^\varepsilon}\le  |\overline{u}|\norm{\mc L[1]}_{H^\varepsilon} + \norm{\mc K[v]}_{H^\varepsilon} \le \norm{u}_{H^{-1}} \paren{\norm{\mc L[1]}_{L^2} + \norm{\mc Q[1]}_{L^2}}+ \norm{\mc K[v]}_{H^\varepsilon} \\
    &\le C \paren{\norm{\bm{X}}_{C^1}+\frac{1}{\starnorm{\bm{X}} } +
    \frac{ \norm{\bm{X}}_{C^{1,\delta}}}{ \starnorm{\bm{X}}}  + \log(\pi/2) }\norm{u}_{H^{-1}} + C \frac{\norm{\bm{X}}_{C^{1,\delta}}^2}{\starnorm{\bm{X}}^2}  \norm{v}_{L^2} \\
    &\le  C \paren{\norm{\bm{X}}_{C^{1,\delta}}+\frac{1}{\starnorm{\bm{X}} } +
    \frac{\norm{\bm{X}}_{C^{1,\delta}}^2}{\starnorm{\bm{X}}^2}  + \log(\pi/2) }\norm{u}_{H^{-1}}.
\end{aligned}
\end{equation}

Then, we consider the map $\mc L:L^2(\T)\to H^{1+\varepsilon}(\T)$.
For $u\in L^2(\T) $, by using~\eqref{eqn:L-l2est} and~\eqref{eqn:KQ_Heps_est},
\begin{align}
    \norm{\mc L[u]}_{L^2} &\le C\paren{\norm{\bm{X}}_{C^1}+\frac{1}{\starnorm{\bm{X}} } + \log(\pi/2)} \norm{u}_{L^2}, \\
    \norm{\partial_s\mc L[u]}_{H^\varepsilon} &= \norm{\mc Q[u]}_{H^\varepsilon}\le C \frac{\norm{\bm{X}}_{C^{1,\delta}}^2}{\starnorm{\bm{X}}^2} \norm{u}_{L^2}.
\end{align}
Therefore, we obtain
\begin{equation}
\begin{aligned}
\norm{\mc L[u]}_{H^{1+\varepsilon}} &\le C \paren{\norm{\bm{X}}_{C^{1,\delta}}+\frac{1}{\starnorm{\bm{X}} } +
    \frac{\norm{\bm{X}}_{C^{1,\delta}}^2}{\starnorm{\bm{X}}^2}  + \log(\pi/2) }\norm{u}_{L^2}.
\end{aligned} 
\end{equation}
For every $s\in [0,1]$, by the interpolation \rev{identities \cite[Appendix~B]{McLean2000}} $H^{-1+s}(\T)=[H^{-1}(\T),L^2(\T)]_{s}$ and $H^{s+\varepsilon}(\T) = [H^\varepsilon(\T), H^{1+\varepsilon}(\T)]_{s}$, we complete the proof.
\end{proof}

Now we are ready to prove Lemma~\ref{lem:vg-decomp}.
By the norm equivalence Lemma~\ref{lem:norm-equv-1}, we set $\psi(s)$ as $(\psi\circ \bm{X})(s)$ and work on the torus $\mathbb T$.
The operator $V_\Gamma$ can be written as
\begin{equation}
\begin{aligned}
    V_\Gamma \psi (s) &= -\frac{1}{2\pi}\int_{\mathbb T} \log| \bm{X}(s) - \bm{X}(t)| \psi(t)\, dt\\
    &= -\frac{1}{2\pi} \int_{\mathbb T} \log \abs{ 2 \sin\frac{s-t}{2} } \psi(t)\, dt -\frac{1}{2\pi} \int_{\mathbb T} \log \abs{ \frac{\bm{X}(s)-\bm{X}(t)}{2 \sin\frac{s-t}{2} }} \psi(t)\, dt \\
    &\coloneqq \mc{L}_1 \psi(s) + \mc{L}_2\psi(s).
\end{aligned}
\end{equation}
The Fourier multiplier for $\mc{L}_1$ is given by
\begin{equation}
    \wh{\mc{L}_1}(k) = -\frac{1}{2\pi}\int_{\mathbb T} \log\abs{2\sin\frac{s}{2}} e^{-iks}\,ds = 
    \begin{cases}
        0, & k = 0,\\
        \frac{1}{2|k|}, & k\in \Z\setminus \{0\}.
    \end{cases}
\end{equation}
Hence, $\mc{L}_1:H^{-\frac{1}{2}}(\mathbb T)\to H^{\frac{1}{2}}(\mathbb T)$ is bounded
\begin{equation}
\begin{aligned}
    \norm{\mc{L}_1 u}_{H^\frac{1}{2}(\mathbb T)}^2 &= \frac{1}{\pi}\sum_{k\neq 0} (1 + |k|^2)^\frac{1}{2} \frac{1}{4 |k|^2} |\wh u_k|^2 \le \frac{1}{4\pi} \sum_{k\neq 0} \frac{1}{(1 + |k|^2)^{\frac{1}{2}}} |\wh u_k|^2\\
    &\le \frac{1}{4\pi} \sum_{k \in \Z} \frac{1}{(1 + |k|^2)^{\frac{1}{2}}} |\wh u_k|^2 = \frac{1}{2} \norm{u}_{H^{-\frac{1}{2}}(\mathbb T)}^2.
\end{aligned}
\end{equation}
Furthermore,
\begin{equation}
\begin{aligned}
    \inprod{u}{\mc{L}_1 u}_{L^2(\mathbb T)} &= \frac{1}{2\pi} \sum_{k\neq 0} \frac{|\wh u_k|^2}{2|k|}  \ge \frac{1}{4\pi} \sum_{k\neq 0}\frac{|\wh u_k|^2}{(1+|k|^2)^{\frac{1}{2}}}\\
    &=\frac{1}{4\pi} \sum_{k\in\Z}\frac{|\wh u_k|^2}{(1+|k|^2)^{\frac{1}{2}}} - \frac{1}{4\pi} |\wh u_0|^2 = \frac{1}{2}\norm{u}_{H^{-\frac{1}{2}}(\mathbb T)}^2 - \inprod{u}{\mc{R} u}_{L^2(\mathbb T)},
\end{aligned}
\end{equation}
where $\mc{R}$ is defined as
\begin{equation}
    \mc{R}u(s) = \frac{1}{2}\wh u_0 = \frac{1}{2}\int_{\mathbb T} u(x)\,dx,
\end{equation}
and it maps any integrable function to a constant function, which belongs to $C^\infty(\mathbb T)$.
Let $V_{\Gamma,0} = \mc{L}_1 + \mc{R}$. The above implies that $V_{\Gamma,0}$ is strictly coercive. 
Let $V_{\Gamma, 1} = \mc{L}_2 - \mc{R}$. 
Since $R$ maps any integrable function to a constant function, which belongs to $H^s(\T)$ for every $s\in \R$, by using Lemma~\ref{lem:L2-op}, we complete the proof of the lemma.

\section*{Acknowledgments}
The authors gratefully acknowledge support from the National Institute for Theory and Mathematics in Biology (NITMB).
The authors thank Vicente Gomez Herrera and Michael Shelley for helpful discussions.
YM was supported by NSF DMR-2309034 (MRSEC) and the Simons Foundation Math+X Chair Fund MPS-MATHX-00234606.

\section*{Conflicts of Interest}
The authors declare no conflicts of interest.

\section*{Data Availability Statement}
Data available on request from the authors.

\bibliographystyle{unsrt}
\bibliography{references}
\end{document}